\documentclass[11pt,letterpaper]{amsart}

\usepackage{epsfig,overpic,comment}
\usepackage[usenames,dvipsnames,svgnames,table]{xcolor}
\usepackage[hyphens]{url}
\usepackage[pagebackref,linktocpage=true,colorlinks=true,linkcolor=Blue,citecolor=BrickRed,urlcolor=RoyalBlue]{hyperref}
\usepackage[msc-links,abbrev]{amsrefs}
\usepackage{amsmath,amsthm,amssymb}
\usepackage{enumerate}
\usepackage[T1]{fontenc}

\usepackage{lmodern}
\usepackage[utf8]{inputenc}

\usepackage{graphicx}
\usepackage{subcaption}

\usepackage{float}
\usepackage{placeins}

\makeatletter 
\AtBeginDocument{%
	\typeout{Text width = \the\textwidth}%
	\typeout{Text height = \the\textheight}%
	} 
	\makeatother

\newtheorem{theorem}{Theorem}[section]

\theoremstyle{definition}
\newtheorem{note}[theorem]{Note}

\newcommand{\KK}{\mathcal{K}}

\newcommand{\RR}{\mathbb{R}}

\newcommand{\HH}{\mathbb{H}}
\newcommand{\SSS}{\mathbb{S}}

\newcommand{\al}{\alpha}

\newcommand{\p}{\partial}

\newcommand{\HHcal}{\mathcal{H}}
\newcommand{\NNN}{\mathcal{N}}

\newtheorem{prop}{Proposition}[section]
\newtheorem{thm}[prop]{Theorem}
\newtheorem{lem}[prop]{Lemma}

\newtheorem*{claim}{Claim}

\theoremstyle{remark}
\newtheorem{rem}[prop]{Remark}
\newtheorem{ques}[prop]{Question}

\numberwithin{equation}{section}
\allowdisplaybreaks
\begin{document}
	
	\setlength{\baselineskip}{1.2\baselineskip}
	
	\title[On Santal\'{o}'s Problem]{On Santal\'{o}'s Problem}
	
	\author[F. Hong]{Fang Hong}
	\address{Department of Mathematics and Statistics, McGill University, Montreal, Quebec, H3A 2K6, Canada.}
	\email{\href{mailto:fang.hong@mail.mcgill.ca}{fang.hong@mail.mcgill.ca}}
	\urladdr{\url{https://sites.google.com/view/fang-hong}}
	
	\begin{abstract}
		This paper investigates an isoperimetric-type problem posed by L. A. Santal\'{o},  concerning convex surfaces in hyperbolic 3-space that minimize total mean curvature among all convex surfaces with fixed surface area. This problem asks for a characterization of the minimizers and for the optimal form of a Minkowski-type inequality in hyperbolic 3-space. In this work, we propose a conjectural description of the minimizers under certain regularity assumptions. We also construct a new family of convex surfaces as as potential minimizer candidates and establish a property of the singular points of any minimizer.
	\end{abstract}  
	
	\thanks{This research was supported by the Dr. and Mrs. Milton Leong Fellowships in Science and an ISM Graduate Scholarship.}
	\maketitle	
	\section{Introduction}
	
	This paper concerns an isoperimetric-type problem raised by L. A. Santal\'{o}. In 1963, Santal\'{o} conjectured \cite{santalo1963} (see also  \cite[p. 78]{santalo2009}) that, for any compact convex domain $\Omega$ in hyperbolic space $\HH^3(a)$ with constant curvature $a \leq 0$ and with $\Gamma=\partial \Omega$,
	\begin{equation}\label{eq:santalo-conj1}
		M(\Gamma) \geq \sqrt{16\pi S(\Gamma)-4a S(\Gamma)^2},
	\end{equation}
	where $S(\Gamma)$ denotes the surface area of $\Gamma$, $M(\Gamma) := \int_\Gamma Hd\mu$ is the \emph{total mean curvature} of $\Gamma$, and the \emph{mean curvature} of $\Gamma$ is the trace of the second fundamental form $H:=\textup{trace}(\mathrm{I\!I}_\Gamma)$.
	The lower bound of \eqref{eq:santalo-conj1} would then correspond to the total mean curvature of a sphere with the same area as $\Gamma$. That is, if \eqref{eq:santalo-conj1} holds, then geodesic balls minimize total mean curvature among convex surfaces with fixed area. 
	
	By scaling the metric, we may take $a=-1$ without loss of generality, and thus the ambient space is the standard hyperbolic $3$-space $\HH^3$. The conjecture was disproved by Naveira-Solanes \cite[p. 815]{santalo2009} (see also \cite[p. 109]{natario2015} or \cite[Note 1.3]{ghomi-spruck2023}). Naveira-Solanes constructed examples indicating that \eqref{eq:santalo-conj1} is false in general. They showed that a flat double disk, which is isometric to a geodesic sphere in a totally geodesic plane isometric to $\HH^2$, embedded in $\HH^3$, with its two faces counted in the surface area and its edge counted in the singular total mean curvature, gives a counterexample to \eqref{eq:santalo-conj1} when the surface area $S(\Gamma)$ is large enough. In fact, if we denote $M_1(S)$ to be the total mean curvature of the geodesic sphere in $\HH^3$ with surface area $S$, and $M_2(S)$ to be the total mean curvature of the flat double disk in $\HH^3$ with surface area $S$, then there exists a constant $\hat{S} = \frac{8\pi(\pi^2-8)}{16-\pi^2} \approx 7.665 $ such that $M_1(S) > M_2(S)$ whenever $S>\hat{S}$. We will explain that this family of convex bodies provides a counterexample to \eqref{eq:santalo-conj1} in Section \ref{sec:1}. On the positive side, Ge-Wang-Wu \cite[Theorem 6.1]{ge-wang-wu2014}, see also \cite{cheng2012}, proved that among horo-convex surfaces with fixed surface area, the minimizer of total mean curvature is the geodesic sphere. That is, \eqref{eq:santalo-conj1} holds for all compact horo-convex surfaces $\Gamma$ in $\HH^3$. 	
	
	In hyperbolic spaces, the total mean curvature of a general convex surface (without any regularity assumption) is well-defined by approximation of outer parallel surfaces, see \cite[Section 3]{ghomi-spruck2023}. By Blaschke selection theorem, a minimizer of total mean curvature among convex domains in hyperbolic space $\HH^3$ with a given surface area exists. One would like to characterize the geometry of the minimizers, as well as the optimal form of Minkowski-type inequality. We call this \textbf{Santal\'{o}'s problem}. This is a generalization of the isoperimetric problem in $\HH^3$ related to the total mean curvature and surface area. 
	
	Minimization of total mean curvature of convex hypersurfaces has long been studied. In 1901, Minkowski \cite{minkowski1903} proved that the following inequality holds for any non-empty bounded
	$C^2$ convex surface $\Gamma \subset \RR^3$, 
	\begin{equation}
		\label{eq:minkowski0}
		M(\Gamma)\geq \sqrt{16\pi S(\Gamma)},
	\end{equation}
	and equality holds only when $\Gamma$ is a sphere in $\RR^3$.
	
	The original proof of Minkowski is based on the isoperimetric inequality together with Steiner-Minkowski formulas. Hence, inequality \eqref{eq:minkowski0} remains true if $\Gamma$ is only a $C^{1,1}$ surface (or equivalently, if $\Gamma$ has positive reach). If we do not assume any regularity, the same inequality holds if we use the mean width of the convex body enclosed by the surface as an alternative definition for total mean curvature. 
	Guan-Li \cite[Theorem 2]{guan-li2009} proved that \eqref{eq:minkowski0} holds provided that $\Gamma$ is star-shaped and mean-convex. G. Huisken \cite[Theorem 6]{guan-li2009} (see also \cite{huisken2009}) showed that \eqref{eq:minkowski0} holds for outward-minimizing surfaces. Dalphin-Henrot-Masnou-Takahashi \cite[Theorem 1.1]{dalphin2016}
	proved \eqref{eq:minkowski0} in the case where $\Gamma$ is axially symmetric and such that $\Gamma \cap P$ is connected for every affine plane $P$ orthogonal to the axis of symmetry.
	
	Inequality \eqref{eq:minkowski0} is actually a consequence of a generalization due to Minkowski of the isoperimetric inequality. This generalization uses the notion of mixed volumes and quermassintegrals of convex bodies, and total mean curvature is in fact the first quermassintegral in Euclidean spaces. We refer to \cite{schneider2014}, \cite{solanes2006}, \cite{santalo1976book} for details, and \cite{wang-xia2014} for quermassintegrals in hyperbolic spaces.
	Generalizations of \eqref{eq:minkowski0} to hyperbolic spaces have been a long-standing problem \cite{santalo1963},  \cites{gallego-solanes,natario2015}. Santal\'{o}'s problem is one such problem. In recent years, using of curvature flows has led to progress, see \cites{wang-xia2014,ge-wang-wu2014,andrews-hu-li-2020,scheuer2020,Brendle-Guan-Li,ghomi-spruck2023}. 
	
	Brendle-Guan-Li \cite{Brendle-Guan-Li} proved that for any compact convex body $\Omega$ in $\HH^3$ with $\Gamma = \p \Omega$,
	\begin{equation}\label{eq:minkowski3-W1}
		W_1(\Gamma) \geq \sqrt{S(\Gamma)}\sqrt{S(\Gamma)+4\pi} + 4\pi \operatorname{arcsinh}\left( \sqrt{\frac{S(\Gamma)}{4\pi}} \right)
		,
	\end{equation}
	where $W_1(\Gamma) = M(\Gamma) - 2 \operatorname{Vol}(\Omega)$ is the first quermassintegral of the convex body $\Omega$ in $\HH^3$ enclosed by $\Gamma$, and \emph{mean convex} means the mean curvature $H$ is non-negative on $\Gamma$. Equality holds only if $\Gamma$ is a geodesic sphere in $\HH^3$. Although \eqref{eq:minkowski3-W1} is sharp in $\HH^3$, the volume of the enclosed domain appears in the inequality. Ghomi-Spruck \cite{ghomi-spruck2023} proved that for any compact convex surface $\Gamma$ in $\HH^3(a)$ with constant curvature $a\leq 0$,
	\begin{equation}\label{eq:minkowski2}
		M(\Gamma)\geq \sqrt{16\pi S(\Gamma)- 2a S(\Gamma)^2},
	\end{equation}
	where equality holds only if $a=0$ and the domain bounded by $\Gamma$ is isometric to a ball in $\RR^3$. Hence \eqref{eq:minkowski2} is not sharp when the ambient space is fixed to be $\HH^3$. 
	Recently the author \cite{hong2026} strengthened \eqref{eq:minkowski2} to a sharp inequality in $\HH^3$. For any compact convex body $\Omega$ in $\HH^3$ with $\Gamma = \p \Omega$,
	\begin{equation}\label{eq:minkowski-hong}
		M(\Gamma) \geq \sqrt{16\pi S(\Gamma) + 2 S(\Gamma)^2 + 2S_0(\operatorname{Vol}(\Omega))^2},
	\end{equation}
	where $S_0$ denotes the {isoperimetric profile function}, that is, $S_0(x)$ is the surface area of the (geodesic) sphere in $\HH^3$ with volume $x$. Equality holds only if $\Gamma$ is a geodesic sphere in $\HH^3$. \eqref{eq:minkowski-hong} appears to be the sharpest Minkowski-type inequality in hyperbolic 3-space $\HH^3$ currently known.
	
	In $\mathbb H^3$,  Blaschke selection theorem guarantees the existence of a minimizer of total mean curvature among convex domains with fixed surface area (we include a brief proof in Section \ref{sec:1}). Yet, the shape of the general convex minimizer is not known. Therefore, \emph{Santal\'{o}'s problem}, on finding the optimal convex surface with the minimum total mean curvature $M$ among convex surfaces with fixed surface area $S$, is still open. We believe a description of the minimizer and the related geometric inequality can deepen our understanding of hyperbolic geometry, and have applications in mathematical physics and general relativity. For example, the total mean curvature in hyperbolic space is also used in the definition of Wang-Yau's quasi-local mass in \cite[Theorem 1.3]{wang-yau2007}.
	
	In the search for the minimizer to Santal\'{o}'s problem, the following was observed by Guan \cite{guan2024}, 
	\begin{prop}\label{thm:santalo-conj2}
		If $\Gamma$ is a convex surface in $\HH^3$ that attains the minimum total mean curvature among surfaces with fixed area, that is, there exists a constant $S>0$ such that
		\begin{align*}
			M(\Gamma) = \min\left\{ M(\Sigma): \Sigma \text{ is a convex surface  in } \HH^3, S(\Sigma) = S\right\},
		\end{align*}
		then
		\begin{enumerate}
			\item If $\Gamma$ is degenerate, that is, the convex hull of $\Gamma$ has no interior point, then $\Gamma$ is a flat double disk.
			\item If $\Gamma$ is $C^2$ and strictly convex, then $\Gamma$ is a geodesic sphere.
		\end{enumerate}
	\end{prop}
	
	Therefore, if one can prove that every minimizer of Santal\'{o}'s problem falls into these two cases, then the minimizer is determined to be either the geodesic sphere or the flat double disk. As a result, Proposition \ref{thm:santalo-conj2} led Guan to speculate the optimal inequality for Santal\'{o}'s problem to be
	\begin{align}\label{eq:santalo-conj2}
		M(\Gamma) \geq 
		&\min\left\{ M_1(S(\Gamma)), M_2(S(\Gamma)) \right\}
		\\
		=&
		\nonumber
		\min\left\{\sqrt{16\pi S(\Gamma) + 4 S(\Gamma)^2},
		\sqrt{2\pi^3 S(\Gamma) + \frac{\pi^2}{4} S(\Gamma)^2}
		\right\},
	\end{align}
	
	However, it turns out that Santal\'{o}'s problem is more complicated. We construct a family of convex bodies in $\HH^3$ which shows \eqref{eq:santalo-conj2} does not hold in general. 
	\begin{thm}\label{thm:counterexample}
		There exist positive constants $S_1, S_2 $ with $S_1< S_2$ such that for any $S$ with $S_1< S < S_2$, there exists a convex surface $\Gamma^{S}$ such that $\Gamma^{S} $ has surface area $S$ and 
		\begin{align*}
			M(\Gamma^S) <
			&\min\left\{ M_1(S), M_2(S) \right\}
			\\
			=&
			\nonumber
			\min\left\{\sqrt{16\pi S + 4 S^2},
			\sqrt{2\pi^3 S + \frac{\pi^2}{4} S^2}
			\right\}.
		\end{align*}
	\end{thm}
	We will show that 
	$\Gamma^{S} = \Gamma(r_S,\alpha_S)$ arises from a two-variable family of convex surfaces, and can be described as a geodesic sphere cut by two totally geodesic planes in $\HH^3$, forming a "drum" shape.
	The existence of such surfaces suggests that the optimal convex surface is possibly non-smooth.
	
	In the study of convex bodies without any regularity assumption on their boundaries, the theory of singular points and curvature measures provides useful tools. General convex surfaces may have singular points, where there are multiple unit outer normal vectors (or supporting planes). If the ambient space is $3$-dimensional, there are two types of singular points: $1$-singular points, with a $2$-dimensional normal cone, behaving like edges of a cube; $0$-singular points, with a $3$-dimensional normal cone, behaving like vertices of a cube. We will give a brief review for singular points and normal cones in Section \ref{sec:1}, Subsection \ref{subsec.prelim.4}.
	Regarding the geometry of singular points of minimizers of Santal\'{o}'s problem, we will prove the following theorem:
	\begin{thm}\label{thm:santalo-sing}
		If $\Gamma$ is a convex surface in $\HH^3$ that attains the minimal total mean curvature with fixed surface area, then $\Gamma$ has no $0$-singular point. That is, for any point $p$ on $\Gamma$, the normal cone $N(p,\operatorname{conv(\Gamma)})$ of the convex body $\operatorname{conv}(\Gamma)$ at $p$ is $1$-dimensional or $2$-dimensional; $p$ is either a $1$-singular point or a regular point.
	\end{thm}
	Theorem \ref{thm:santalo-sing} shows that even if the minimizer of Santal\'{o}'s problem is non-smooth and singular, it cannot have points that behave like the vertex of a cone or the vertex of a triangular pyramid.  This theorem can help to rule out some convex surfaces as minimizers of Santal\'{o}'s problem, such as polytopes or cones in $\HH^3$. 
	
	

	The rest of this paper is organized as follows. In Section \ref{sec:1}, we list some notations and facts about the known examples, and review the background in convex geometry. In Section \ref{sec:2}, we will prove Proposition \ref{thm:santalo-conj2}. In Section \ref{sec:3}, we will construct the new family of convex bodies whose total mean curvature is strictly less than those of the geodesic sphere and the flat double disk with the same area, thereby forming a family of counterexamples to \eqref{eq:santalo-conj2} and proving Theorem \ref{thm:counterexample}. In Section \ref{sec:4}, we will prove Theorem \ref{thm:santalo-sing} using a local argument based on estimating the total mean curvature of a new convex body truncated by a family of planes at the singular point. We will discuss open questions and further study directions related to Santal\'{o}'s problem at the end of this paper.
	
	\section{Preliminaries}\label{sec:1}
	
	In this paper, any ambient manifold $N$ is an $(n+1)$-dimensional space form. A \emph{convex set} $X$ in $N$ is a set whose convex hull $\operatorname{conv}(X)$ is itself. A \emph{convex hypersurface} is the boundary of a convex set. If the ambient manifold $N$ has dimension $\dim(N) = 3$, then we call the convex hypersurface a \emph{convex surface}. A \emph{convex body} is a bounded closed convex set.
	
	When a hypersurface is $C^2$, we have the following equivalent definition for convexity: A \emph{convex} hypersurface $\Gamma$ of an ambient manifold $N$ is a closed embedded submanifold of codimension one which, when properly oriented, 
	has positive semidefinite second fundamental form $\mathrm{I\!I}_\Gamma$. A \emph{strictly convex} hypersurface $\Gamma$ of $N$ is a convex hypersurface with positive definite second fundamental form.
	
	\subsection{Naveira-Solanes' counterexample}\label{subsec.prelim.1}
	Here we will list some facts about the geometry of hyperbolic 3-space $\HH^3$ and examine Naveira-Solanes' counterexample \cite{natario2015} mentioned above to show that the flat double disks form a counterexample to \eqref{eq:santalo-conj1}. See also \cite[Note 1]{ghomi-spruck2023}.
	
	For a geodesic sphere $\SSS(r)$ with radius $r$ in $\HH^3$,
	it is well known that its enclosed volume $V$,  surface area $S$ and total mean curvature $M$ are
	\begin{align*}
	V = V(B(r)) &= 2\pi 
	\left( \sinh(r) \cosh(r) - r \right) ,
	\\
	\nonumber
	S = S(\SSS(r)) &= 4\pi \sinh(r)^2 ,
	\\
	\nonumber
	M = M(\SSS(r)) &= 8\pi \sinh(r) \cosh(r) ,
	\end{align*}
	and in particular, we have
		$M = \sqrt{16\pi S + 4S^2}$,
	and we may denote this function by $M_1$, that is,
	\begin{align}\label{defn.M1}
		M_1(S) := \sqrt{16\pi S + 4S^2}.
	\end{align}
	
	The flat double disk in Naveira-Solanes' counterexample \cite{natario2015} mentioned above is constructed by taking a 
	disc $D(r)$ of radius $r$ in a totally geodesic surface isometric to $\HH^2$ in hyperbolic space $\HH^3$. Its total mean curvature can be computed by taking the limit of the total mean curvature of its outer parallel surface $\Gamma = \Gamma(r,\epsilon)$ at distance $\epsilon$. Its surface area consists of two faces of the disk $D(r)$:
	$$
	S = 2 S_{\HH^2}(D)=4\pi(\cosh(r)-1),
	$$
	where the $S$ on the left-hand side denotes the total surface area of the flat double disk $D$ in $\HH^3$ and the  $S_{\HH^2}(D)$ denotes the area of the disk with radius $r$ in $\HH^2$.
	
	Note that $\Gamma$ consists of a pair of topological disks parallel to $D$ plus a half tube $T$ about $\partial D$. The mean curvature of the disks vanishes as $\epsilon\to 0$. On the other hand, $S(T) \to L(\partial D)\pi\sinh(\epsilon)$ up to first order (where $L(\p D)$ denotes the perimeter of $\p D$ in $\HH^2$), since the full tube about $\partial D$ is foliated by (geodesic) circles of radius $\epsilon$. Therefore $ M(\Gamma)\to \partial S(T)/\partial\epsilon=L(\partial D) \pi\cosh(\epsilon)$. Thus
	$$
	M = \lim_{\epsilon\to 0} M(\Gamma)=L(\partial D)\pi
	=2\pi^2\sinh(r).
	$$ 
	Thus for the flat double disk, its total mean curvature $M = M(D(r))$ and surface area $S = S(D(r))$ satisfies
	\begin{align*}
		M = \frac{\pi}{2}\sqrt{8\pi S+ S^2} =  \sqrt{ 2\pi^3  S + \frac{\pi^2}{4}S^2}.
	\end{align*}
	We denote this function by $M_2$, that is,
	\begin{align}\label{defn.M2}
	M_2(S) := \sqrt{ 2\pi^3 S + \frac{\pi^2}{4}S^2}.
	\end{align}
	
	
	Clearly when $S$ is large enough, $M_2(S) < M_1(S)$. In fact, $M_2(S) < M_1(S)$ for all $S > \hat{S} = \frac{8\pi(\pi^2-8)}{16-\pi^2} \approx 7.665$. 
	
	In this paper, we will use the Beltrami-Klein model for hyperbolic 3-space. Specifically, we view $\HH^3$ as $ (B_{\RR^3}(1) , g^{BK})$. For $x = (x_1,x_2,x_3) \in B_{\RR^3}(1) \subset \RR^3$, the Riemannian metric tensor on $T_{x}\HH^3$ is 
	$$ g^{BK}(x) = \frac{1}{1-\sum x_i^2} \left( \sum dx_i\otimes dx_i \right) + \frac{1}{(1-\sum x_i^2)^2} 
	\left( \left(\sum x_i dx_i\right) \otimes \left(\sum x_i dx_i\right)  \right), $$
	where all sums are taken as $\sum_{i=1}^{3}$.
	
	The following coordinate identity map gives a canonical homeomorphism from $\HH^3$ to $ B_{\RR^3}(1)$.
	\begin{align}\label{defn.BK.identity}
		\operatorname{Id}:(B_{\RR^3}(1) , g^{BK}) \rightarrow& (B_{\RR^3}(1), g^{\RR^3})\\
		\nonumber
		(x_1,x_2,x_3) \mapsto& (x_1,x_2,x_3).
	\end{align}
	
	An advantage of the Beltrami-Klein model is that geodesics are preserved. That is, for a continuous curve $\gamma = \gamma(t)$ in $\HH^3 =(B_{\RR^3}(1) , g^{BK}) $, $\gamma$ is a geodesic in  $(B_{\RR^3}(1) , g^{BK})$ if and only if $\operatorname{Id}(\gamma)$ is a geodesic in $(B_{\RR^3}(1), g^{\RR^3})$ (but not necessarily of the same speed).
	
	\subsection{Notations and Facts about Geometric Flows}\label{subsec.prelim.2}
	Here we list some evolution equations for geometric flows and recall outer parallel surfaces.
	
	A \emph{geometric flow}  of a $C^2$ closed hypersurface $\Gamma$ in a Riemannian $(n+1)$-manifold $N$ \cite{andrews-chow2020,giga2006,huisken-polden1999} is a one-parameter family of immersions $X\colon\Gamma\times[0,T)\to N$, $X_t(\cdot):=X(\cdot, t)$,  given by
	\begin{equation}\label{eq:hmcf}
		X_t'(p)=-F_t(p)\nu_t(p),\quad\quad\quad X_0(p)=p,
	\end{equation}
	where $(\cdot)':=\partial/\partial t(\cdot)$,
	$\nu_t$ is a  normal vector field along $\Gamma_t:=X_t(\Gamma)$, and the \emph{speed function} $F_t$ is a smooth function on $\Gamma_t$. Let $\nu_t(p)$ be the normal and $\kappa_i^t(p)$ the principal curvatures of $\Gamma_t$ at the point $X_t(p)$. 
	Let $d\mu_t$ be the area element induced on $\Gamma$ by $X_t$. $G_t:=\det(\mathrm{I\!I}_t)$ and $H_t:=\textup{trace}(\mathrm{I\!I}_t)$ are the \emph{Gauss-Kronecker curvature} and \emph{mean curvature} of $\Gamma_t$, respectively.
	
	By \cite[Thm. 3.2(v)]{huisken-polden1999} and \cite[Lem. 7.4]{huisken-polden1999}, for any geometric flow,
	\begin{align}
		\label{eq:evolution}
		\frac{d}{dt}(H_t) &= 
		\Delta_t F_t+\left(|\mathrm{I\!I}_t|^2+\operatorname{Ric}(\nu_t)\right)F_t,
		\\
		\nonumber
		\frac{d}{dt}(d\mu_t) &= -F_tH_td\mu_t,
		\\
		\nonumber
		\frac{d}{dt} S(\Gamma_t) &= -\int_{\Gamma} F_t H_t d\mu_t,
	\end{align}
	where $|\mathrm{I\!I}_t|:=\sqrt{\sum(\kappa_i^t)^2}$, $\Delta_t$ is the Laplace-Beltrami operator induced on $\Gamma$ by $X_t$, and $\operatorname{Ric}(\nu_t)$ is the Ricci curvature of $N$ at the point $X_t(p)$ in the direction of $\nu_t(p)$, i.e., the sum of sectional curvatures of $N$ with respect to a pair of orthogonal planes which contain $\nu_t(p)$.
	
	
	In particular, when the speed function $F_t \equiv 1$, the geometric flow exists for all $t>0$ and is called \emph{geodesic normal flow}. The hypersurface $\Gamma_t$ obtained from $\Gamma$ under geodesic normal flow is called the \emph{outer parallel surface} of $\Gamma$ at distance $t$. An equivalent definition of outer parallel surface is given by $\Gamma_t: = \widehat{d_{\Gamma}}^{-1}(t)$, where $\widehat{d_{\Gamma}}$ is the distance function of the convex surface $\Gamma$, see for example \cite[Sec. 2 \& 3]{ghomi-spruck2022}. Hence the outer parallel surface is well-defined for any convex surface as an initial surface.
	
	A fact which will be used frequently below is that for any convex surface $\Gamma$ in $\HH^3$, its outer parallel surface $\Gamma_t$ is $C^{1,1}$ and convex for any $t>0$ \cite[Sec. 2 \& 3]{ghomi-spruck2022}. In particular, for $t>0$, $\Gamma_t$ is twice differentiable almost everywhere and so its total mean curvature $M(\Gamma_t)$ is well-defined and positive.
	We will prove the following invariant along geodesic normal flow:
	\begin{prop}\label{parallel.surface.property}
		Let $\Gamma$ be a convex surface in $\HH^3$, and $\Gamma_t$ be its outer parallel surface at distance $t$. Then for any $t\geq 0$, 
		\begin{align}\label{eq:parallel.surface.property}
			M(\Gamma_t)^2 - 16\pi S(\Gamma_t) - 4S(\Gamma_t)^2
			=
			M(\Gamma)^2 - 16\pi S(\Gamma) - 4S(\Gamma)^2
		\end{align}
	\end{prop}
	\begin{proof}
		For any $t>0$, since $\Gamma_t$ is twice differentiable almost everywhere, we may apply \eqref{eq:evolution} 
		 and plug in $F_t\equiv 1$ to get
		\begin{align}
			\nonumber
			\frac{d}{dt} M(\Gamma_t) 
			=&
			\int_{\Gamma}\left(\frac{d}{dt}(H_t)d\mu_t+H_t\frac{d}{dt}(d\mu_t)\right)\\ 
			\nonumber
			=&
			\int_{\Gamma}\Big(\Delta_t F_t+\big(|\mathrm{I\!I}_t|^2-(H_t)^2\big)F_t +\operatorname{Ric}(\nu_t)F_t\Big)d\mu_t
			\\
			\nonumber
			=
			&
			2
			\int_{\Gamma} (G_t + 1 ) d\mu_t
			\\
			\label{eq:evolution-parallel}
			\frac{d}{dt} S(\Gamma_t)
			=
			&
			\int_{\Gamma} H_t d\mu_t
			=
			M(\Gamma_t).
		\end{align}
		By Gauss' equation, for all $p\in\Gamma_t$, 
		\begin{equation*}
			G_t(p)=K_{\Gamma_t}(p)-K_N(T_p\Gamma_t),
		\end{equation*}
		where $K_{\Gamma_t}$ is the sectional curvature of $\Gamma_t$, and $K_N(T_p\Gamma_t)$ is the sectional curvature of $N$ with respect to the tangent plane $T_p \Gamma_t\subset T_p N$.   Thus by Gauss-Bonnet theorem, 
		\begin{equation*}
			\int_{\Gamma} G_t d\mu_t=4\pi -\int_{p\in\Gamma_t} K_N(T_p\Gamma_t)
			=
			4\pi + S(\Gamma_t).
		\end{equation*}
		Substituting this into \eqref{eq:evolution-parallel}, we have for any $t>0$,
		\begin{align*}
			\frac{d}{dt} \left( M(\Gamma_t)^2 - 16\pi S(\Gamma_t) - 4S(\Gamma_t)^2 \right) = 0,
		\end{align*}
		which implies that the quantity
		$M(\Gamma_t)^2 - 16\pi S(\Gamma_t) - 4S(\Gamma_t)^2$ is a constant for all $t\geq 0$.
	\end{proof}
	
	\subsection{Quermassintegrals of convex bodies in hyperbolic spaces}\label{subsec.prelim.3}
	
	Here we list definitions and properties of quermassintegrals in Euclidean spaces and hyperbolic spaces.
	One may refer to \cite{solanes2006}, \cite{santalo1976book} and \cite{wang-xia2014} for more background.
	
	For a (geodesically) convex domain $K$ in a $(n+1)$-dimensional space form $N$ of constant curvature $C \in \{ -1,0,1 \}$, an equivalent definition of \emph{quermassintegrals} is the following:
	\begin{align}\label{defn.quermass}
		&W_{k-1}(K):=\frac{(n+1-k)\omega_{k-1}\cdots\omega_0}{\omega_{n-1}\cdots\omega_{n-k}}\int_{\mathcal{L}_k(N)}\chi(L_k\cap K)dL_k(N),\\
		\nonumber
		& \quad k=1,\cdots,n,
	\end{align}
	where $\omega_{k}$ denotes the Lebesgue measure of the $k$-dimensional unit
	sphere $\SSS^k$. Here $\mathcal{L}_k(N)$ is the space of $k$-dimensional totally geodesic subspaces $L_k$ in $N$ and $dL_k(N)$ is the natural (invariant) measure on $\mathcal{L}_k(N)$. The function $\chi$ is given by $\chi(K)=1$  if $K\neq \emptyset$ and $\chi(\emptyset)=0.$
	For simplicity, we also use the convention $W_{-1}(K)=V(K)$ for the volume of $K$, and $W_{n}(K)={\omega_{n-1}}$.
	By definition, 
	$W_{0}(K)=S(\p K)$.
	
	In Euclidean spaces, if $\p K$ is a $C^2$ hypersurface, the quermassintegrals coincide with the definition in terms of curvature integrals (integral of elementary symmetric functions of curvatures on $\p K$). However, the quermassintegrals and the curvature integrals in hyperbolic spaces or spheres do not coincide. Nevertheless they are closely related. In a space form $N$ with dimension $(n+1)$ and constant curvature $C \in \{ -1,0,1 \}$, if $\p K$ is a $C^2$ hypersurface, then
	we have
	\begin{align}\label{defn.quermass.1}
			W_1(K)=&\int_{\p K} H d \mu_g+ n C V(K) = M(\p K) + n C V(K), \\
			\label{defn.quermass.general.k}
			W_k(K)=&\int_{\p K} \sigma_k(\kappa) d \mu_g+\frac{C(n-k+1)}{k-1} W_{k-2}(K),  \quad \text{ for } 2\leq k \leq n, 
	\end{align}
	Since total mean curvature is also well-defined for general convex bodies (or surfaces) in Cartan-Hadamard spaces without any regularity assumption (see e.g. \cite[Section 3]{ghomi-spruck2023}), \eqref{defn.quermass.1} also holds for a general convex body $K$.
	
	Combining \eqref{defn.quermass} and \eqref{defn.quermass.1}, and plugging $n=2$, $k=1$, and $C=0$ or $C=-1$ respectively, we have
	\begin{itemize}
		\item If $K$ is a convex body in $\HH^3$, then 
		\begin{align}\label{quermass.formula.hyperbolic}
			M(\p K)  =  W_1(K) + 2 V(K) = \int_{\mathcal{L}_2(\HH^3)}\chi(L_2\cap K)dL_2(\HH^3),
		\end{align}
		\item If $K$ is a convex body in $\RR^3$, then
		\begin{align}\label{quermass.formula.euclidean}
			M(\p K)  =  W_1(K) = \int_{\mathcal{L}_2(\RR^3)}\chi(L_2\cap K)dL_2(\RR^3).
		\end{align}
	\end{itemize}
	$dL_2(N)$ is induced by the action of the isometry group of $N$ on $\mathcal{L}_2(N)$. One may refer to \cite[Chapter 4.4]{schneider2014} for more details. To estimate total mean curvature using the formulas above, we now give an explicit form of the invariant measure $d L_2(N)$ for $N = \HH^3$ or $N = \RR^3$. After taking an origin $O$ in $\RR^3$, we may parameterize $\mathcal{L}_2(\RR^3)$ as follows:
	\begin{align*}
		\mathcal{L}_2(\RR^3) = \left\{ 
		\operatorname{Plane}(\nu , R):\nu \in \SSS^2, R>0
		 \right\},
	\end{align*}
	where $\operatorname{Plane}(\nu, R)$ is the plane in $\RR^3$ with unit outer normal $\nu$, and the distance from $\operatorname{Plane}(\nu, R)$ to $O$ is $R$, or more precisely,
	\begin{align}\label{grassmannian.paramter.R}
		\operatorname{Plane}(\nu, R): = \left\{ p\in \RR^3: \left\langle p , \nu \right\rangle = R \right\},
	\end{align}
	where $ \left\langle \cdot , \cdot \right\rangle $ is the canonical inner product in $\RR^3$.
	Under this parametrization, the invariant measure can be formulated as
	\begin{align}\label{invariant.measure.RR3}
		dL_2(\RR^3) = 2 d R \wedge d\mu_{\SSS^2}(\nu),
	\end{align}
	where $dR$ is the Lebesgue measure on $R\in \RR$, and $d\mu_{\SSS^2} $
	is the Lebesgue measure on $\nu\in \SSS^2$.
	
	The parametrization of $L_2(\HH^3)$ follows similarly, and we will use similar notations. After taking an origin $O$ in $\HH^3$, we may parameterize $\mathcal{L}_2(\HH^3)$ as follows:
	\begin{align*}
		\mathcal{L}_2(\HH^3) = \left\{ 
		\operatorname{Plane}(\nu , R):\nu \in \SSS^2, R>0
		\right\},
	\end{align*}
	where $\operatorname{Plane}(\nu, R)$ is a plane (totally geodesic $2$-dimensional submanifold) in $\HH^3$ that that is analogous to its counterpart in $\RR^3$. In the Beltrami-Klein model where the origin is taken as the same origin $O$, we may formulate it as:
	\begin{align}\label{grassmannian.paramter.H}
		\operatorname{Plane}(\nu, R): = \operatorname{Id}^{-1} \left( \operatorname{Plane}(\nu, \tanh(R)) \cap B_{\RR^3}(1) \right)
	\end{align}
	where $\operatorname{Id}$ is the Beltrami-Klein identity map given by \eqref{defn.BK.identity}, and the $\operatorname{Plane}(\nu, \tanh(R))$ on the right-hand side is a plane in $\RR^3$ defined in \eqref{grassmannian.paramter.R}. Note that the distance between the hyperbolic plane $\operatorname{Plane}(\nu, R)$ defined in \eqref{grassmannian.paramter.H} and the origin $O$ is still $R$. Under this parametrization, the invariant measure can be formulated as
	\begin{align}\label{invariant.measure.HH3}
		dL_2(\HH^3) = 2 (1 + 2 \sinh^2(R)) d R \wedge d\mu_{\SSS^2}(\nu).
	\end{align}

	\subsection{Convexity and Normal Cones} \label{subsec.prelim.4}
	Here we review the definitions and facts about the normal cone and singular (regular) points, together with basics about curvature measures used in convex geometry. One may refer to \cite{schneider2014} for definitions and facts in Euclidean spaces. The counterparts in space forms are similar.
	
	A \emph{supporting hyperplane} $P$ of a closed set $X$ is a totally geodesic hypersurface in $N$ such that $X$ is entirely contained in one of the two closed half-spaces bounded by the hyperplane $P$, and $P \cap \p X \neq \emptyset$. We say that $P$ is a supporting hyperplane of $X$ at $p$, if $P$ is supporting hyperplane of $X$, and $p \in P \cap \p X $. If the ambient manifold $N$ has dimension $\dim(N) = 3$, then we call the supporting hyperplane a \emph{supporting plane}. 
	
	If $K$ is a convex body, and $p \in \p K$, then clearly there exists a supporting hyperplane $P$ of $K$ at $p$. The supporting hyperplane at a point $p$ may not be unique. The set of all supporting hyperplanes at $p\in \p K$ with respect to the convex body $K$ is denoted by $\operatorname{Supp}(p,K)$.
	
	Conversely, if $K$ is a closed set with non-empty interior such that every point $p$ on the boundary has a supporting hyperplane $P$, then $K$ is a convex body and $K$ is the intersection of all its supporting closed half-spaces (namely the half-space bounded by $P=P(p)$ that contains K).
	$X$ is a {strictly convex set} if every supporting hyperplane of $X$ intersects $X$ at only a single point.
	
	In the following context of this subsection, we take the ambient manifold $N$ to be the hyperbolic 3-space $\HH^3$.
	
	If $\operatorname{Supp}(p,K)$ consists of only one element (one plane), we say $p$ is a \emph{regular} point of the convex surface $\p K$. In this case, there exists a unique supporting plane of $p\in \p K$ with respect to $K$. Otherwise, we say $p$ is a \emph{singular} point of $\p K$. See \cite[Chapter 2.2]{schneider2014}.
	If $p$ is singular, let $r=r(p)$ be the maximum number of linearly independent supporting planes at $p$ with respect to $K$. Since $\dim (T_p \HH^3) = 3 $, we know $r=2$ or $r=3$. We call $p$ a $(3-r)$\emph{-singular point}. 
	
	We denote the set of regular points on $\p K$ by $\operatorname{Reg}(\p K)$. For $s=0,1$, we denote the set of $s$-singular points on $\p K$ by $ \operatorname{Sing}^{s}(\p K) $. 
	
	If $K$ is a convex body in $\HH^3$, for any $p\in \p K$ and any supporting plane $P\in \operatorname{Supp}(p,K)$, $P$ can be determined by its non-zero normal vector $\nu(p,H) \in T_p \HH^3$, that is, $\nu(p,H)$ is the unique unit vector up to orientation such that
	\begin{align*}
		P  = \left\{ \exp_{p}(v) : v\in T_p\HH^3, \bar{g}\left(v,\nu(p,H)\right) = 0 \text{ in } T_p\HH^3 \right\},
	\end{align*}
	where $\bar{g}$ is the Riemannian metric of $\HH^3$.
	We take such a vector $\nu$ to be outward with respect to $K$, that is, $\exp_{p}(t\nu)$ lies in $\HH^3 \backslash K$ for any $t>0$. 
	
	We define the \emph{normal cone} $N(p,K)$ of $p \in \p K$ with respect to the convex body $K$ as the set of outer normal vectors (not necessarily unit vectors) of all supporting planes of $p$ with respect to $K$. That is,
	 \begin{align}\label{defn.notation.unit.normal}
	 	N(p,K) = \{O\} \cup \left\{\lambda \nu(p,H) : \lambda>0, H \in \operatorname{Supp}(p,K) \right\} \subset T_p \HH^3. 
	 \end{align}
	
	Note that the vectors in a normal cone are not necessarily unit vectors. We denote the set of unit normal vectors at a point $p\in \p K$ by $\mathcal{N}(p,K)$, that is,
	\begin{align*}
		\mathcal{N}(p,K) := \left\{ \nu \in N(p,K) : \|\nu \| =1 \right\}.
	\end{align*}
	The unit normal bundle $\NNN(K)$ is defined by $\NNN(K) := \bigcup_{p\in \p K} \NNN(p,K)$.
	In this paper, we will also use the notation $\operatorname{relint}\NNN(p,K)$, which means the relative interior of $\NNN(p,K)$ in its affine hull in $\{ \nu \in T_p\HH^3: \|\nu\| = 1 \}$, which is isometric to $\SSS^2$.
	
	\begin{claim}
		$N(p,K)$ is a convex cone in $(T_p \HH^3, \bar{g})$, where $\bar{g}$ denotes the Riemannian metric in $\HH^3$.
	\end{claim} 
	
	\begin{proof}[Proof of the Claim]
		It suffices to prove that for any non-zero vectors $\nu_1, \nu_2 \in N(p,K)$, and any $\lambda \in (0,1)$, $\lambda \nu_1 + (1-\lambda) \nu_2 \in N(p,K)$. If $\lambda \nu_1 + (1-\lambda) \nu_2=0$, the claim follows directly. Hence we assume $\lambda \nu_1 + (1-\lambda) \nu_2 \neq 0$ in the following proof.
		
		Clearly if $H$ is a supporting plane of $K$ in $\HH^3 =(B_{\RR^3}(1) , g^{BK}) $, then after being mapped into $B_{\RR^3}(1) \subset \RR^3$ as in \eqref{defn.BK.identity}, $\operatorname{Id}(H)$ is a supporting plane of $\operatorname{Id}(K)$ in  $(B_{\RR^3}(1), g^{\RR^3})$, where we choose the Beltrami-Klein model so that $ \operatorname{Id}(p) $ is the origin in $ (B_{\RR^3}(1), g^{\RR^3}) $. Moreover, $ \operatorname{Id}(K) $ is also a convex body in $ (B_{\RR^3}(1), g^{\RR^3}) $.
		
		By definition, there exist two totally geodesic planes $H_1, H_2$ in $\HH^3$ such that $\nu_1 = \nu(p,H_1), \nu_2 = \nu(p,H_2)$. Consider $ \operatorname{Id}(H_1) $ and $ \operatorname{Id}(H_2) $ as planes in $ (B_{\RR^3}(1), g^{\RR^3}) $ passing through the origin.
		Consider $ \overline{H}_3 = \lambda \operatorname{Id}(H_1) + (1-\lambda) \operatorname{Id}(H_2) $. Since $ \operatorname{Id}(K) $ is convex, and both 
		$ \operatorname{Id}(H_1) $ and $ \operatorname{Id}(H_2) $ support $ \operatorname{Id}(K) $, $ \overline{H}_3 $ is also a supporting plane of $ \operatorname{Id}(K) $, hence $ H_3 := \operatorname{Id}^{-1} (\overline{H_3}) $ is a supporting plane of $ K $.
		
		Since $ \nu(p,H_3) = \frac{1}{\| \lambda \nu_1 + (1-\lambda) \nu_2 \|}\left(\lambda \nu_1 + (1-\lambda) \nu_2\right) $, we have $ \lambda \nu_1 + (1-\lambda) \nu_2 \in N(p,K) $.
	\end{proof}
	
	If $p$ is regular, then $N(p,K)$ is $1$-dimensional. In this case, there exists a unique unit outer normal vector in $T_p \HH^3$ at $p$. Otherwise, if $p$ is singular, then $N(p,K)$ is at least $2$-dimensional. 
	Since $\dim (T_p \HH^3) = 3$, we know if $p$ is singular, then $\dim N(p,K)=2$ or $\dim N(p,K) = 3$. If $p$ is a $0$-singular point, then $\dim N(p,K) = 3$. If $p$ is a $1$-singular point, then $\dim N(p,K) = 2$.   

	\section{Proof of Proposition \ref{thm:santalo-conj2}: Minimizers under special conditions}\label{sec:2}
	
	In this section, we will first prove the existence of a minimizer in Santal\'{o}'s problem. Then we will prove Proposition \ref{thm:santalo-conj2}, that is, the classification of minimizers in two cases: degenerate case and $C^2$ strictly convex case. The proofs of Proposition \ref{thm:santalo-conj2} follow from Guan \cite{guan2024}.
	
	\subsection{Existence of a minimizer}
	
	For any fixed real number $S_0 \in (0,\infty)$, consider the set of convex bodies defined by $\KK^3(S_0) = \left\{ K\in \KK^3(\HH^3) : \p K \text{ has surface area } S_0 \right\} $, where $\KK^3(\HH^3)$ denotes the set of all convex bodies in $\HH^3$.
	By Ghomi-Spruck \cite{ghomi-spruck2023}, for any convex surface $\Gamma$ in $\HH^3$, we have
	$$ M(\Gamma) \geq \sqrt{16\pi S(\Gamma)+ 2S(\Gamma)^2 }, $$
	hence we know for any $K\in \KK^3(S_0)$,
	$ M(\p K) \geq  \sqrt{16\pi S_0+ 2S_0^2 }, $
	thus
	$$ \inf \left\{ M(\p K) : K\in \KK^3(S_0) \right\} \geq \sqrt{16\pi S_0+ 2S_0^2 } >0. $$
	
	We denote $M_0 := \inf \left\{ M(\p K) : K\in \KK^3(S_0) \right\}$. 
	Consider a sequence of convex bodies $K_n$ such that 
	$\lim_{n \rightarrow \infty} M(\p K_n) = M_0 $. We denote $\Gamma_n:=\p K_n$. Without loss of generality, we assume $M(\Gamma_n) \leq M_0 + 1$ for all $n$.
	
	By the inequality relating diameter and total mean curvature in hyperbolic space \cite{wu-zheng2011}, there exists a constant $C$ such that for any convex closed connected embedded surface $\Gamma$ in $\HH^3$, 
	$$ \operatorname{diam}(\Gamma) \leq C \int_{\Gamma} H d\mu, $$
	where $\operatorname{diam}(\cdot)$ denotes the diameter of a set in $\HH^3$.
	Hence for all $n$, since $\p K_n$ are convex surfaces,
	$$ \operatorname{diam}(K_n) = \operatorname{diam}(\Gamma_n) \leq C \int_{\Gamma_n} H d\mu = C \cdot M(\Gamma_n) \leq C(M_0+1). $$
	Therefore by applying isometric translations to these convex bodies appropriately, we may assume for all $n$, $K_n$ is contained in the geodesic ball $B( C(M_0+1) )$ with radius $ C(M_0+1) $ in $\HH^3$. 
	
	By Blaschke Selection Theorem in metric spaces, see for example \cite{gruber2007}, any bounded sequence of convex bodies in $\HH^3$ contains a convergent subsequence with respect to the Hausdorff metric. Hence there exists a convex body $K_{\infty}$ and a sequence $m_n \rightarrow \infty$, such that
	$K_{m_n}$ tends to $K_{\infty}$ with respect to Hausdorff metric.
	
	Hence $S_0 =\lim_{n \rightarrow \infty} S(\p K_{m_n}) = S(\p K_{\infty}) $, and $M_0 = \lim_{n \rightarrow \infty} M(\p K_{m_n}) = M(\p K_{\infty})  $, that is, $K_{\infty} \in \KK^3(S_0)$, and
	$ M(\p K_{\infty}) = \inf \left\{ M(\p K) : K\in \KK^3(S_0) \right\}. $
	That is, $\p K_{\infty}$ is the minimizer of total mean curvature of convex surfaces with surface area $S_0$.
	
	\subsection{Degenerate minimizer}
	
	\begin{proof}[Proof of Proposition \ref{thm:santalo-conj2}(1)]
		If $\Gamma$ is degenerate, then by convexity, $\Gamma$ is a subset of a totally geodesic submanifold isometric to $\HH^2$ in $\HH^3$, with its two faces counted into surface area, and its edge counted into singular mean curvature.
	
	Denote this convex body in $\HH^2$ by $K_1$. Then we have
	\begin{align*}
		S(\Gamma) = 2 S(K_1), \quad M(\Gamma) = \pi L(\p K_1),
	\end{align*}
	where $S(K_1)$ and $L(\p K_1)$ denote area of $K_1$ and perimeter of $\p K_1$ in $\HH^2$, respectively. Then Santal\'{o}'s problem in this case is equivalent to the isoperimetric problem in $\HH^2$. By \cite{weil1940}, the convex body in $\HH^2$ that minimizes perimeter with fixed area is the geodesic ball in $\HH^2$. Hence $\Gamma$ is the flat double disk in $\HH^3$.
	\end{proof}
	
	\subsection{$C^2$ strictly convex minimizer}
	
	Fix the surface area $S$. If a convex surface $\Gamma$ is the minimizer of total mean curvature, and a piece of $\Gamma$ is $C^2$ and strictly convex, that is, both of its principal curvatures are positive in an open subset $U$ of $\Gamma$, then we may apply the method of Lagrange multipliers to derive the Euler-Lagrange equation of the surface.
	Given any normal variation of the surface $\Gamma$ by a geometric flow, say $\p_t \Gamma_t = f_t \nu_t$, with $\Gamma_0 =\Gamma$ and $f_0=f$, where $f$ is a function on $\Gamma$ compactly supported in $U$ which is a relatively open subset of $\Gamma$, we have
	at time $t=0$, the first variation of total mean curvature and area are given by
	\begin{align*}
		&\p_t M = \int_{U} 2(G+1) f d\mu,
		\\
		&\p_t S = \int_{U} H f d\mu .
	\end{align*}
	One may refer to Section \ref{sec:2}, Subsection \ref{subsec.prelim.2} for details.
	
	Since both principal curvatures of $\Gamma$ are positive in $U$, there exists $\epsilon>0$ such that for any $t$ with $-\epsilon < t < \epsilon$, $\Gamma_{t}$ is still a $C^2$ strictly convex surface in $\HH^3$.
	
	Therefore, as a critical point of the functional $M$, $\Gamma$ satisfies the Euler-Lagrange equation 
	$$2(G+1) = \tilde{\lambda} H \text{ in } U,$$
	for a fixed real number $\tilde{\lambda}$.
	By denoting $\lambda = \tilde{\lambda}/2$, we have
	$$G+1 = \lambda H \text{ in } U.$$
	By convexity, $\lambda$ is positive.
	If the whole surface of the minimizer $\Gamma$ is $C^2$ and strictly convex, the equation above holds globally on $\Gamma$.
	
		\begin{proof}[Proof of Proposition \ref{thm:santalo-conj2}(2)]
			Since the minimizer $\Gamma$ is $C^2$ and strictly convex, we may derive the Euler-Lagrange equation for the entire surface from the arguments above, hence there exists a constant $\lambda>0$, such that
		\begin{equation}\label{e-l-gamma}
			G+1 = \lambda H \text{ in } \Gamma.
		\end{equation}
		Integrating both sides of the equation on $\Gamma$, and applying the Gauss-Bonnet Theorem $\int_{\Gamma} G d\mu = 4\pi + S$ in $\HH^3$, we obtain 
		\begin{equation}\label{S-M-relation1}
			4 \pi + 2S = \lambda M.
		\end{equation}
		Since $\Gamma$ is the minimizer of total mean curvature, we know its total mean curvature is no more than that of a sphere with the same area, that is,
		\begin{align}\label{sphere-compare}
			M \leq M_1(S) := \sqrt{16\pi S + 4S^2},
		\end{align}
		therefore
		$$\lambda = \frac{4\pi + 2S}{M} \geq \frac{4\pi+ 2S}{\sqrt{16\pi S + 4S^2}} = 
		\sqrt{\frac{16\pi^2 + 16\pi S + 4S^2 }{16\pi S + 4S^2} }>1.$$
		
		Hence we may write \eqref{e-l-gamma} as
		$$\det(\kappa-\lambda I)=(\kappa_1-\lambda)(\kappa_2-\lambda) = \lambda^2-1>0.$$
		Therefore, $(\kappa-\lambda I)$ is either positive-definite everywhere or negative-definite everywhere.
		Now we claim: Every point in $\Gamma$ is umbilical. We will first prove $(\kappa-\lambda I)$ is positive-definite everywhere.
		
		Suppose $(\kappa-\lambda I)$ is negative-definite everywhere, that is, $(\lambda I-\kappa)$ is positive-definite everywhere. Then we have for every point in $\Gamma$,
		$$ 2\lambda - H = \operatorname{tr}(\lambda I -\kappa) \geq 2 \sqrt{\det(\lambda I -\kappa)} = 2 \sqrt{\lambda^2-1}. $$
		Integrating over $\Gamma$, we have
		$$ 2\lambda S -M \geq 2 \sqrt{\lambda^2-1}S, $$
		that is,
		\begin{equation}\label{temp3.1}
			M \leq 2 S (\lambda -\sqrt{\lambda^2-1}) = 2S \frac{1}{\lambda +\sqrt{\lambda^2-1}}
		\end{equation}
		
		From \eqref{S-M-relation1} we get $\lambda = {(4\pi +2 S)}/{M}$. Substituting this into \eqref{temp3.1}, we have
		\begin{align*}
			M \leq 2 S \frac{M}{4\pi + 2 S + \sqrt{(4\pi +2 S)^2-M^2}} 
			\leq 2 S \frac{M}{4\pi + 2 S}.
		\end{align*}
		Since $M >0$, we now get 
		$$1 \leq \frac{2 S}{4\pi + 2S},$$
		that is, $4\pi + 2 S \leq 2 S$, which leads to a contradiction.
		Therefore, $(\kappa-\lambda I)$ cannot be negative-definite everywhere.  $(\kappa-\lambda I)$ is positive-definite everywhere. Now we prove $\Gamma$ is umbilical everywhere.
		
		Now we have
		$$ H - 2\lambda = \operatorname{tr}(\kappa-\lambda I ) \geq 2 \sqrt{\det(\kappa-\lambda I )} = 2 \sqrt{\lambda^2-1}. $$
		Integrate over $\Gamma$, we have
		\begin{align}\label{temp3.2}
			M \geq 2 S (\lambda + \sqrt{\lambda^2-1}), 
		\end{align}
		where the equality holds if and only if $\Gamma$ is umbilical everywhere.
		Again from \eqref{S-M-relation1} we get $\lambda = {(4\pi +2 S)}/{M}$. Substituting this into \eqref{temp3.2}, we get
		\begin{align*}
			M \geq 
			&2 S \left( \frac{4\pi+2 S}{M} + \sqrt{\left( \frac{4\pi+2 S}{M} \right)^2-1} \right) 
			\\
			= &2 S \left( \frac{4\pi+2 S}{M} + \frac{1}{M}\sqrt{\left( {4\pi+2 S} \right)^2-M^2} \right)
		\end{align*}
		Now using \eqref{sphere-compare} again, we have
		\begin{align*}
			M \geq& 2 S \left( \frac{4\pi+2 S}{M} + \frac{1}{M}\sqrt{\left( {4\pi+2 S} \right)^2-(16\pi  S+ 4S^2)} \right)
			\\
			=&2 S  \frac{1}{M} \left( 4\pi+2 S + 4\pi \right),
		\end{align*}
		that is,
		\begin{align*}
			M^2 \geq 16\pi S + 4S^2,
		\end{align*}
		that is,
		\begin{align}\label{sphere-compare2}
			M \geq M_1(S).
		\end{align}
		Combining \eqref{sphere-compare2} and \eqref{sphere-compare}, we obtain $M = M_1(S)$ and the equality in \eqref{temp3.2} holds, hence $\Gamma$ is umbilical everywhere. That is, $\kappa_1 = \kappa_2 = \lambda >1$ everywhere. Thus $\Gamma$ is a sphere in $\HH^3$.
		\end{proof}

	\section{Proof of Theorem \ref{thm:counterexample}: A new family of convex surfaces}\label{sec:3}
	
	In this subsection, we will construct a two-variable family of convex surfaces $\Gamma(r,\alpha)$ that forms a counterexample to \eqref{eq:santalo-conj2} and thereby proves Theorem \ref{thm:counterexample}.
	We construct $\Gamma(r,\al)$ as follows. Take a sphere $\mathbb{S}(r)$ in $\HH^3$ with radius $r$, and truncate it by two parallel totally geodesic planes, both of which are isometric to $ \HH^2 $. More precisely, $\Gamma(r,\alpha) := \p \Omega(r,\alpha)$, where $\Omega(r,\alpha)$ is the following convex body in $\HH^3$ defined via the Beltrami-Klein model (the origin can be taken arbitrarily in advance, and then fixed):
	\begin{align*}
		&
		\Omega(r,\alpha): = 
		\\
		&\operatorname{Id}^{-1}\left( (x_1,x_2,x_3) \in B_{\RR^3}(1): 
		(x_1)^2 + (x_2)^2 + (x_3)^2 \leq \hat{r}^2, 
		-\hat{r}\cos(\alpha) \leq x_3 \leq \hat{r}\cos(\alpha)
		 \right),
	\end{align*}
	where the identity map is defined in \eqref{defn.BK.identity}, and $\hat{r} = \tanh(r)$ is the Euclidean radius of the image under $\operatorname{Id}$ of the sphere $\mathbb{S}(r)$ in $\HH^3$ with radius $r$. See also Figure \ref{fig:fig2}.
	
	\begin{figure}[h]
		\centering
		\includegraphics[width=0.5\textwidth]{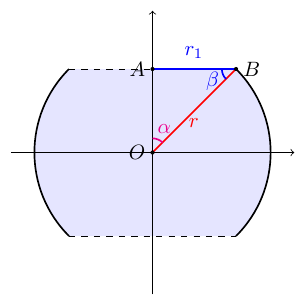}
		\caption{The section of the set $\Omega(r,\alpha)$ is filled blue in the Beltrami-Klein model. Note that the distances and angles are measured in the hyperbolic metric, hence the shape is distorted.}
		\label{fig:fig2}
	\end{figure}
	
	We denote the following three points by $O := \operatorname{Id}^{-1}(0,0,0)$, $A := \operatorname{Id}^{-1}(0,0,\hat{r}\cos(\alpha))$ and $B:= \operatorname{Id}^{-1}(\hat{r}\sin(\alpha),0,\hat{r}\cos(\alpha))$.
	In the hyperbolic right triangle $OAB$ (as shown in Figure \ref{fig:fig2}), we will introduce the following notation. Clearly by definition of $\Omega(r,\alpha)$, we have the hypotenuse $\operatorname{dist}_{\HH^3}(O,B) = r = \operatorname{arctanh}(\hat{r})$, and the angle $ \angle AOB = \alpha$. By the hyperbolic law of sines, $r_1: = \operatorname{dist}_{\HH^3}(A,B) $ is given by $ \sinh(r_1) = \sinh(r) \sin(\al) $. By the hyperbolic law of cosines, the angle $\beta: = \angle ABO$ is given by $\cot(\beta) \cot(\alpha) = \cosh(r)$. 
	
	We decompose the surface $\Gamma(r,\alpha)$ into three pairwise disjoint subsets. Namely $\Gamma(r,\alpha) = \Gamma_1(r,\alpha)\cup \Gamma_2(r,\alpha)\cup \Gamma_3(r,\alpha) $, where 
	\begin{align*}
		&\Gamma_1(r,\alpha) : = 
		\\
		&\operatorname{Id}^{-1}\left\{
		 (x_1,x_2,x_3) \in B_{\RR^3}(1): 
	0\leq (x_1)^2 + (x_2)^2 \leq \hat{r}^2, 
	\right.
	\\
	&
	\left.
	x_3 = \hat{r}\cos(\alpha) \text{ or } -\hat{r}\cos(\alpha) 
	\right\}
	\end{align*}
	denotes the upper and lower "flat" parts, which are isometric to two identical disjoint disks embedded in $\HH^2$ with radius $r_1$;
	\begin{align*}
		&\Gamma_2(r,\alpha) 
		: =
		\\
		& \operatorname{Id}^{-1}\left\{ 
		(x_1,x_2,x_3) \in B_{\RR^3}(1): 
		(x_1)^2 + (x_2)^2 +(x_3)^2 = \hat{r}^2, \right.
		\\
		&
		\left.
		-\hat{r}\cos(\alpha) < x_3 < \hat{r}\cos(\alpha) 
		\right\}
	\end{align*}
	denotes the "spherical" part, which is isometric to a subset of the geodesic sphere $\SSS^2(r)$ with radius $r$ in $\HH^3$. The solid angle of the spherical part $\Gamma_2(r,\alpha)$ in $\SSS^2(r)$ is $4\pi -2\cdot 2\pi (1-\cos(\alpha)) = 4\pi \cos(\alpha)$, hence its area proportion of the whole sphere $\SSS^2(r)$ is $\cos(\alpha)$; 
	\begin{align*}
		&\Gamma_3(r,\alpha) : =
		\\
		& \operatorname{Id}^{-1}\left\{ (x_1,x_2,x_3) \in B_{\RR^3}(1): 
		(x_1)^2 + (x_2)^2 + (x_3)^2 = \hat{r}^2, 
		\right.
		\\
		&
		\left.
		x_3 = \hat{r}\cos(\alpha) \text{ or } -\hat{r}\cos(\alpha) 
		\right\}
	\end{align*} 
	denotes the two "edges" where $ \Gamma_1(r,\alpha)$ and $\Gamma_2(r,\alpha)$ meet.

	The surface area of $\Gamma(r,\alpha)$ is the sum of that of $\Gamma_1(r,\alpha)$ and $\Gamma_2(r,\alpha)$. The area of $\Gamma_2(r,\alpha)$ is the corresponding proportion of area of the sphere in $\HH^3$ with radius $r$, hence $S(\Gamma_2(r,\alpha)) = 4\pi \sinh^2(r) \cos(\alpha) $. That is,
	\begin{align*}
		S= S(\Gamma(r,\al)) &= 
		4\pi \sinh^2(r) \cos(\alpha)
		+ 4\pi (\cosh(r_1)-1).
	\end{align*}
	
	To compute 
	the total mean curvature of $\Gamma(r,\alpha)$, following the Steiner formula in hyperbolic space \cite{kohlmann1991}, we will use the notion of mean curvature measure, see Appendix \ref{append.mean.curv.compute}. To be specific, 
	\begin{align*}
		&M(\Gamma(r,\alpha)) = 
		\\
		&
		\Phi_1(\Omega(r,\alpha), \Gamma_1(r,\alpha))
		+
		\Phi_1(\Omega(r,\alpha), \Gamma_2(r,\alpha))
		+
		\Phi_1(\Omega(r,\alpha), \Gamma_3(r,\alpha)),
	\end{align*}
	where 
	\begin{itemize}
		\item $\Phi_1(\Omega(r,\alpha), \Gamma_1(r,\alpha))$ denotes the total mean curvature on $\Gamma_1(r,\alpha)$, that is, the "flat" part, hence $\Phi_1(\Omega(r,\alpha), \Gamma_1(r,\alpha)) = 0$ because both components of $\Gamma_1(r,\alpha)$ lie in two totally geodesic planes of $\HH^3$ respectively; 
	\item $\Phi_1(\Omega(r,\alpha), \Gamma_2(r,\alpha))$ denotes the total mean curvature on $\Gamma_2(r,\alpha)$, that is, the "spherical" part, and hence equals the corresponding proportion of the total mean curvature of a sphere in $\HH^3$ with radius $r$, so $$\Phi_1(\Omega(r,\alpha), \Gamma_2(r,\alpha)) = 8\pi \sinh(r) \cosh(r) \cos(\alpha); $$ 
	\item $\Phi_1(\Omega(r,\alpha), \Gamma_3(r,\alpha))$ denotes
	the total mean curvature on the two "edge" components, and equals $$\Phi_1(\Omega(r,\alpha), \Gamma_3(r,\alpha)) =   2\left( \frac{\pi}{2} -\beta \right) L_1 = 4\pi \left( \frac{\pi}{2} -\beta \right)  \sinh(r_1), $$ where $\left( \frac{\pi}{2} -\beta \right)$ is the angle between outer normal vectors of $\Gamma_1(r,\alpha)$ and $\Gamma_2(r,\alpha)$ at any point of $\Gamma_3(r,\alpha)$, and $L_1$ is the length ($1$-dimensional Hausdorff measure $\HHcal^1$) of any component of $\Gamma_3(r,\alpha)$, given by the perimeter of a geodesic ball in $\HH^2$. We will give a detailed computation of $\Phi_1(\Omega(r,\alpha), \Gamma_3(r,\alpha))$ in Appendix \eqref{append.mean.curv.compute}.
	\end{itemize}
	Hence the total mean curvature of $\Gamma(r,\al)$ is:
	\begin{align*}
	M= M(\Gamma(r,\al)) &= 
	8\pi \sinh(r) \cosh(r) \cos(\alpha) 
	+ 4\pi \sinh(r_1) \left( \frac{\pi}{2} -\beta \right) .
	\end{align*}
	Rewriting the terms involving $r_1$ and $\alpha$ in terms of $ r $ and $ \alpha $,
	we have
	\begin{align*}
		 4\pi \sinh(r_1) \left( \frac{\pi}{2} -\beta \right) =& 
		4\pi \sinh(r) \sin(\al) \arctan\left( \cosh(r) \tan(\al) \right),\\
		 4\pi (\cosh(r_1)-1) =&
		4\pi \left( \sqrt{1+\sinh^2(r) \sin^2(\al)}-1 \right).
	\end{align*}
	Thus we obtain the total mean curvature and the surface area of $\Gamma(r,\alpha)$ in terms of $r$ and $\alpha$.
	\begin{align}
		\label{eq1.counterexample}
		M = M(\Gamma(r,\alpha)) =& 8\pi \sinh(r) \cosh(r) \cos(\alpha) 
		\\
		\nonumber
		&+ 4\pi \sinh(r) \sin(\al) \arctan\left( \cosh(r) \tan(\alpha) \right),
		\\
		\label{eq2.counterexample}
		S = S(\Gamma(r,\alpha)) =& 4\pi \sinh^2(r) \cos(\alpha) 
		+ 4\pi \left( \sqrt{1+\sinh^2(r) \sin^2(\alpha)}-1 \right).
	\end{align}
	
	We will now take a specific value of the pair $(r_0,\alpha_0)$ and show that  
	\begin{align}\label{conterexample.final.compare}
		M(\Gamma(r_0,\alpha_0)) < \min\left\{M_1(S(\Gamma(r_0,\alpha_0))), M_2(S(\Gamma(r_0,\alpha_0)))\right\},
	\end{align}
	where $M_1$ and $M_2$ are total mean curvatures of spheres and flat double disks, respectively, as functions of areas defined in Section \ref{sec:1}, Subsection \ref{subsec.prelim.1}.  
	The values of $r_0$ and $\alpha_0$ are the unique values such that the following equations hold.
	\begin{align}
		\label{counterexample.r0.a0.eq1}
		\beta_0 := &\arctan\left( \cosh(r_0) \tan(\alpha_0) \right) = \frac{\pi}{4},
		\\
		\label{counterexample.r0.a0.eq2}
		S(\Gamma(r_0,\alpha_0)) = &4\pi \sinh^2(r_0) \cos(\alpha_0) 
		+ 4\pi \left( \sqrt{1+\sinh^2(r_0) \sin^2(\alpha_0)}-1 \right)
		\\
		\nonumber
		 = &\hat{S} = \frac{8\pi(\pi^2-8)}{16-\pi^2},
	\end{align}
	where $\hat{S}$ is the unique positive real number such that $M_1(\hat{S}) = M_2(\hat{S})$.
	By \eqref{counterexample.r0.a0.eq1}, we have $\cosh(r_0) \tan(\alpha_0) = \tan\left(\frac{\pi}{4}\right)=1$. We denote $\cos(\alpha_0) = k_0$. Therefore we obtain
	\begin{align*}
		\cosh(r_0) = \cot(\alpha_0) = \frac{k_0}{\sqrt{1-k_0^2}} > 1.
	\end{align*}
	Note that this implies that $\frac{1}{\sqrt{2}} < k_0 <1$, as well as
	\begin{align*}
		\sinh(r_0) = \frac{\sqrt{2k_0^2-1}}{\sqrt{1-k_0^2}}, \quad \sin(\alpha_0) = \sqrt{1-k_0^2}.
	\end{align*}
	Hence
	\begin{align*}
		\sqrt{1+\sinh^2(r_0) \sin^2(\alpha_0)} = \sqrt{2}k_0.
	\end{align*}
	Substituting \eqref{counterexample.r0.a0.eq2}, we obtain
	\begin{align*}
	\frac{k_0(2k_0^2-1)}{1-k_0^2}
	+
	\sqrt{2}\,k_0
	-1
	= \frac{S(\Gamma(r_0,\alpha_0))}{4\pi}
	= \frac{2(\pi^2-8)}{16-\pi^2}.
	\end{align*}
	Equivalently,
	\begin{align*}
		(2k_0^2-1)\frac{k_0}{1-k_0^2} + \sqrt{2}k_0 = \frac{\pi^2}{16-\pi^2},
	\end{align*}
	or
	\begin{align}\label{counterexample.eq.k0}
		(16-\pi^2)\sqrt{2}(\sqrt{2}-1)k_0^3
		+
		\pi^2 k_0^2
		+
		(16-\pi^2)(\sqrt{2}-1)k_0
		-
		\pi^2
		=
		0.
	\end{align}
	
	Now we compute \(M(\Gamma(r_0,\alpha_0))\). Substituting \eqref{eq1.counterexample}, we have
	\begin{align}\label{counterexample-M-k0}
		\nonumber
		&M(\Gamma(r_0,\alpha_0))
		\\
		\nonumber
		=&
		8\pi \sinh (r_0) \cosh (r_0) \cos(\alpha_0)
		+
		4\pi \sinh (r_0) \sin(\alpha_0)
		\arctan(\cosh(r_0) \tan(\alpha_0))
		\\
		=&
		8\pi \frac{\sqrt{2k_0^2-1}}{\sqrt{1-k_0^2}} \frac{k_0}{\sqrt{1-k_0^2}} k_0
		+
		4\pi \frac{\sqrt{2k_0^2-1}}{\sqrt{1-k_0^2}} \sqrt{1-k_0^2}
		\frac{\pi}{4}
		\\
		=&
		\sqrt{2k_0^2-1}
		\left(
		\frac{8\pi k_0^2}{1-k_0^2}
		+
		\pi^2
		\right),
  	\end{align}
	where we have used $\arctan(\cosh(r_0) \tan(\alpha_0)) = \frac{\pi}{4}$ as in \eqref{counterexample.r0.a0.eq1}.
	
	We will now show that $M(\Gamma(r_0,\alpha_0)) < M_1(\hat{S}) $. By \eqref{counterexample-M-k0}, we compute
	\begin{align*}
		&M(\Gamma(r_0,\alpha_0))^2 - M_1(\hat{S})^2
		\\
		=
		&(2 k_0^2 - 1)\left(
		\frac{8\pi k_0^2}{1-k_0^2}
		+
		\pi^2
		\right)^2
		-
		M_1\left(\frac{8\pi(\pi^2-8)}{16-\pi^2}\right)
		\\
		=
		&(2 k_0^2 - 1)\left(
		\frac{8\pi k_0^2}{1-k_0^2}
		+
		\pi^2
		\right)^2 - 
		\frac{64\pi^4(2\pi^2 - 16)}{(16-\pi^2)^2}
	\end{align*}

	
	We now prove that
	\[
	M(\Gamma(r_0,\alpha_0))<M_1(S_0).
	\]
	
	Define
	\[
	F(x):=
	\frac{x(2x^2-1)}{1-x^2}
	+
	\sqrt{2}\,x
	-1.
	\]
	Then \(k_0 \in (\frac{1}{\sqrt{2}},1)\) is determined by
	\[
	F(k_0)=\frac{2(\pi^2-8)}{16-\pi^2}.
	\]
	On \(x\in(1/\sqrt{2},1)\), we have
	\[
	F'(x)
	=
	\frac{-2x^4+5x^2-1}{(1-x^2)^2}
	+\sqrt{2}
	>0.
	\]
	Hence \(F\) is strictly increasing.
	
	Let
	\[
	q:=0.78732.
	\]
	Using the elementary estimates
	\[
	1.41421<\sqrt{2},
	\qquad
	\pi<3.1416,
	\]
	we obtain
	\[
	F(q)>0.609997,
	\]
	whereas
	\[
	\frac{2(\pi^2-8)}{16-\pi^2}
	<
	\frac{2(3.1416^2-8)}{16-3.1416^2}
	<0.609966.
	\]
	Thus
	\[
	F(q)>\frac{2(\pi^2-8)}{16-\pi^2}=F(k_0).
	\]
	Since \(F\) is strictly increasing, it follows that
	\[
	k_0<q.
	\]
	
	Define
	\[
	G(x):=
	\sqrt{2x^2-1}
	\left(
	\frac{8\pi x^2}{1-x^2}
	+\pi^2
	\right).
	\]
	Then \(G\) is strictly increasing on \((1/\sqrt{2},1)\). Therefore
	\[
	M(\Gamma(r_0,\alpha_0))=G(k_0)<G(q).
	\]
	We estimate \(G(q)\). Since
	\[
	q^2=0.6198727824,
	\]
	we have
	\[
	2q^2-1<0.239746,
	\qquad
	\frac{q^2}{1-q^2}<1.630699.
	\]
	Using again \(\pi<3.1416\), we obtain
	\[
	\begin{aligned}
		G(q)^2
		&<
		0.239746
		\left(
		8\cdot 3.1416\cdot 1.630699
		+
		3.1416^2
		\right)^2
		\\
		&<620.007
		<
		24.900^2.
	\end{aligned}
	\]
	Hence
	\[
	G(q)<24.900.
	\]
	
	On the other hand,
	\[
	M_1(S_0)^2
	=
	16\pi S_0+4S_0^2
	=
	\frac{128\pi^4(\pi^2-8)}{(16-\pi^2)^2}.
	\]
	Let
	\[
	H(y):=\frac{128y^2(y-8)}{(16-y)^2}.
	\]
	Then \(H\) is strictly increasing on \((8,16)\). Since \(\pi>3.141592\), we have
	\[
	\pi^2>9.8696.
	\]
	Therefore
	\[
	M_1(S_0)^2
	=
	H(\pi^2)
	>
	H(9.8696)
	>
	620.269
	>
	24.905^2.
	\]
	Thus
	\[
	M_1(S_0)>24.905.
	\]
	
	Combining the estimates, we conclude
	\[
	M(\Gamma(r_0,\alpha_0))
	=
	G(k_0)
	<
	G(q)
	<
	24.900
	<
	24.905
	<
	M_1(S_0).
	\]
	Therefore
	\[
	M(\Gamma(r_0,\alpha_0))<M_1(S_0) = M_1(S(\Gamma(r_0,\alpha_0))) = M_2(S(\Gamma(r_0,\alpha_0))),
	\]
	hence \eqref{conterexample.final.compare} holds for 
	$(r,\alpha) = (r_0,\alpha_0)$.

	Since $M(\Gamma(r,\alpha))$, $ M_1(S(\Gamma(r,\alpha))) $ and $ M_2(S(\Gamma(r,\alpha)))$ are all smooth functions of $r$ and $\alpha$ for $r>0$ and $0 < \alpha < \frac{\pi}{2}$, we have \eqref{conterexample.final.compare} also holds for all $(r,\alpha)$ in a neighborhood of $(r_0,\alpha_0)$. Also note that $S(\Gamma(r,\alpha))$ is strictly increasing with respect to $r$, hence Theorem \ref{thm:counterexample} follows.
	
	\begin{note}
		
		Although the "drum" shape $\Gamma(r,\alpha)$ constructed above gives a contradiction to \eqref{eq:santalo-conj2} for a certain range of parameter values of $r$ and $\alpha$, this two-variable family of convex surfaces $\Gamma(r,\alpha)$ does not contradict \eqref{eq:santalo-conj2} when either the area $S$ is close to $0$ or the area $S$ is large enough.
		
		In fact, $\Gamma(r,\alpha)$ only contradicts \eqref{eq:santalo-conj2} and serves as a candidate for a minimizer when the area is very close to $\hat{S}= \frac{8\pi(\pi^2-8)}{16-\pi^2}\approx 7.665$, which is the positive solution to $M_1(S) = M_2(S)$. See Figure \ref{fig:fig3.combined} for samples of $\Gamma(r,\alpha)$.

		\begin{figure}[h]
			\centering
			
			\begin{subfigure}[t]{0.48\textwidth}
				\centering
				\includegraphics[width=\linewidth]{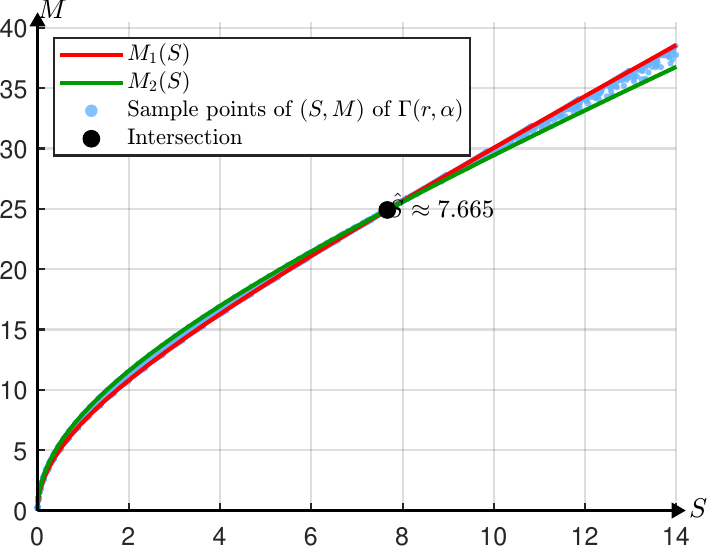}
				\caption{Range for $0< S <14$}
				\label{fig:fig3}
			\end{subfigure}
			\hfill
			\begin{subfigure}[t]{0.48\textwidth}
				\centering
				\includegraphics[width=\linewidth]{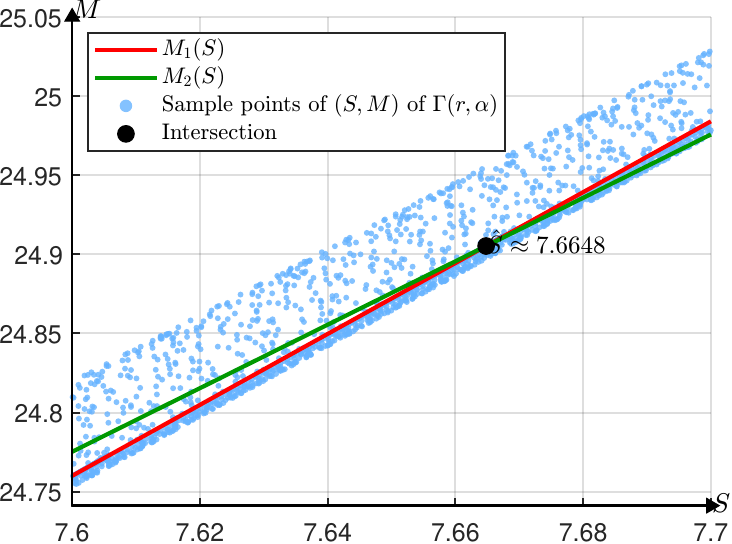}
				\caption{Zoom near $S\approx \hat S \approx 7.665$}
				\label{fig:fig3_zoom}
			\end{subfigure}
			
			\caption{Comparison of $M_1(S)$ and $M_2(S)$ with sampled $(S,M)$ points of $\Gamma(r,\alpha)$.}
			\label{fig:fig3.combined}
		\end{figure}
	\end{note}

	
	\section{Proof of Theorem \ref{thm:santalo-sing}: Singularity property}\label{sec:4}
	In this section, we will prove Theorem \ref{thm:santalo-sing}.
	One may refer to Section \ref{sec:1} for definitions and facts about singular points, and integral geometric formulas for quermassintegrals.
	
	We will use a local argument. The idea is that if there exists a $0$-singular point on the minimizer surface $\Gamma$, then it forms a "vertex" of the whole convex body $\Omega = \operatorname{conv}(\Gamma)$. Hence we can shrink $\Omega$ infinitesimally by truncating it with a family of planes at the "vertex". After the infinitesimal shrinking, we will then perform an infinitesimal expansion using outward geodesic normal flow (outer parallel surface) to restore the area. Finally, we construct a convex surface $\Gamma'$ with the same area as $\Gamma$, and we will prove $\Gamma'$ has strictly less total mean curvature than $\Gamma$, which contradicts the assumption that $\Gamma$ is a minimizer. In the whole proof, we will use the Beltrami-Klein model to describe the infinitesimal shrinking in Euclidean space.
	
	\begin{proof}[Proof of Theorem \ref{thm:santalo-sing}]
		Let $\Gamma$ be a minimizer to Santal\'{o}'s problem. Let $\Omega = \operatorname{conv}(\Gamma)$.
		
		If $\Gamma$ is degenerate, then by Proposition \ref{thm:santalo-conj2}(1), it is a flat double disk, whose two open disks are regular points, and the points along the circular edge are $1$-singular points. Hence the theorem follows. In what follows, we will assume $\Gamma$ is non-degenerate, that is, its convex hull has non-empty interior. 
	
	Suppose that there exists a $0$-singular point $p$ on $\Gamma$. Then the normal cone $N(p, \Omega)$ has at least three linearly independent unit vectors. We choose $P_1,P_2,P_3\in \operatorname{Supp}(p,\Omega)$ to be three linearly independent supporting planes of $\Omega$ at $p$. Let $U_1$ be the closed half-space of $\HH^3$ divided by $P_1$ that contains $\Omega$, and define $U_2$, $U_3$ similarly for $P_2$, $P_3$, respectively. Since $P_1$, $P_2$ and $P_3$ are linearly independent, $\Omega_1 := U_1\cap U_2\cap U_3$ is a closed unbounded triangular pyramid, and $\Omega \subset \Omega_1$. We denote $\Gamma_{1}:= \p \Omega_{1}$.
	
	
	For any unit normal vector $\nu\in \operatorname{relint} {\NNN(p,\Omega)}$ (where $\operatorname{relint}\NNN(p,\Omega)$ means the relative interior of $\NNN(p,\Omega)$ in its affine hull in $\{ \nu \in T_p\HH^3: \|\nu\| = 1 \}$,  as stated in Section \ref{sec:1}, Subsection \ref{subsec.prelim.4}), we define $p(\nu,d)$ to be the point $\exp_{p}(-d\nu)$, and define $\nu(d)$ to be the image of $\nu$ parallel transported from $T_p \HH^3$ to $T_{p(\nu,d)} \HH^3$ for $d>0$, and $P(\nu,d)$ to be the parallel-transported hyperbolic plane. That is,
	\begin{align*}
		P(\nu,d) :=  \left\{ \exp_{p(\nu,d)}(v) : v\in T_{p(\nu,d)}\HH^3, \bar{g}\left(v,\nu(d)\right) = 0 \text{ in } T_{p(\nu,d)}\HH^3 \right\}.
	\end{align*}
	Clearly the distance from $p$ to $P(\nu,d)$ is $d$ in $\HH^3$.
	
	For any unit normal vector $\nu$ in $\operatorname{relint} {\NNN(p,\Omega)}$, and for $d>0$ small enough, $\Omega$ and $\Omega_1$ will both be divided by $P(\nu, d)$  into two convex subsets. The proof will focus on the infinitesimal change of surface area and total mean curvature of the truncated convex surface. The specific choice of $\nu\in \operatorname{relint} {\NNN(p,\Omega)}$ used to perform this cut is to be determined later, and we will set up the notations and lemmas first.
	
	We denote the two closed half-spaces of $\HH^3$ divided by $P(\nu, d)$ to be $P^{+}(\nu,d)$ and $P^{-}(\nu, d)$, where $P^{+}(\nu,d)$ is the one containing $p$, and $P^{-}(\nu, d)$ is the other half-space. We denote $\Omega_{1}^{\pm}(\nu,d) : = \Omega_{1}\cap P^{\pm}(\nu,d)$, and $\Omega^{\pm}(\nu,d) :=  \Omega\cap P^{\pm}(\nu,d)$. Note that all of them are closed convex sets, while $\Omega_{1}^{+}$, $\Omega^{+}$, and $\Omega^{-}$ are bounded, and $\Omega_{1}^{-}$ is unbounded. For the corresponding boundary surfaces, we denote $\Gamma_{1}^{\pm}(\nu,d):= \p \Omega_{1}^{\pm}(\nu,d)$, and $\Gamma^{\pm}(\nu,d):= \p \Omega^{\pm}(\nu,d)$. 
	
	Let $Q_1(\nu,d):= \Omega_{1}\cap P(\nu,d)$ and $Q(\nu,d):= \Omega\cap P(\nu,d)$. Each of $Q_1(\nu,d)$ and $Q(\nu,d)$ is isometric to a subset of $\HH^2$. Then we have $ \Gamma_{1}^{+}(\nu,d) = \left( \Gamma_{1} \cap P^{+}(\nu,d) \right)\cup Q_{1}(\nu,d)$ and $ \Gamma^{+}(\nu,d) = \left( \Gamma \cap P^{+}(\nu,d) \right)\cup Q(\nu,d)$.
	
	
	
	\begin{lem}\label{lem.diff.area}
		For any $\nu\in \operatorname{relint} {\NNN(p,\Omega)}$, there exists constant $C$ depending only on $\Omega, p, \nu$,  such that for $d>0$ small enough, 
		\begin{align*}
			S(\Gamma) - S(\Gamma^{-}(\nu,d)) \leq C d^2 .
		\end{align*}
	\end{lem}
	\begin{proof}[Proof of Lemma \ref{lem.diff.area}]
		By definition, $ S(\Gamma) - S(\Gamma^{-}(\nu,d)) = S(\Gamma\cap P^{+}(\nu,d) ) $. Since $ \Gamma^{+}(\nu,d) = \left( \Gamma \cap P^{+}(\nu,d) \right)\cup Q(\nu,d)$, it suffices to show $S(\Gamma^{+}(\nu,d)) = \mathcal{O}(d^2)$ as $d \to 0^{+}$. 
		By comparison theorem for areas of nested convex surfaces \cite[Lemma 3.4]{ghomi-spruck2023}, since $ \Omega^{+}(\nu,d) \subset \Omega_{1}^{+}(\nu,d) $, we have $S(\Gamma^{+}(\nu,d)) \leq S(\Gamma_{1}^{+}(\nu,d))$. Hence it suffices to show $S(\Gamma_{1}^{+}(\nu,d)) = \mathcal{O}(d^2)$ as $d \to 0^{+}$.
		
		$\Omega_{1}$ is a triangular pyramid in $\HH^3$, thus $\Omega_{1}^{+}(\nu,d)$ is a tetrahedron in $\HH^3$ with diameter $\operatorname{diam}( \Omega_{1}^{+}(\nu,d) ) \leq C d$ for some constant $C$ depending only on $N(p,K)$, for $d>0$ small enough. Hence by classical area formulas of hyperbolic triangles, $S(\Gamma_{1}^{+}(\nu,d)) \leq C d^2 $  for some constant  $C$  depending only on  $\Omega, p, \nu$.
	\end{proof}
	
	\begin{lem}\label{lem.diff.total.mean.curv}
		For any $\nu\in \operatorname{relint} {\NNN(p,\Omega)}$, there exists constant $c$ depending only on $\Omega, p, \nu$,  such that for $d>0$ small enough,
		\begin{align}\label{temp.sing.hyperbolic1}
			M(\Gamma) - M(\Gamma^{-}(\nu,d)) > cd.
		\end{align}
	\end{lem}
	
	\begin{proof}[Proof of Lemma \ref{lem.diff.total.mean.curv}]
		By \eqref{quermass.formula.hyperbolic}, 
	\begin{align}\label{sing.M.difference.hyperbolic}
		M(\Gamma) - M(\Gamma^{-}(\nu,d)) = \int_{\mathcal{L}} d L_2(\HH^3) + 2 V(\Omega^{+}(\nu,d)),
	\end{align}
	where $\mathcal{L}$ denotes the set of all the totally geodesic planes in $\HH^3$ that intersect $\Omega$ but do not intersect $\Omega^{-}(\nu,d)$,  and $d L_2(\HH^3)$ is the natural (invariant) measure on $\mathcal{L}_2(\HH^3)$. See Section \ref{sec:1}, Subsection \ref{subsec.prelim.3}.
	
	For convenience, we will first prove a similar result in Euclidean space $\RR^3$. The proof in $\HH^3$ will then follow. That is, we will prove
	\begin{align}\label{temp.sing.euclidean1}
		M(\operatorname{Id}(\Gamma)) - M(\operatorname{Id}(\Gamma^{-}(\nu,d))) > c d. 
	\end{align}
	for some constant $c >0$ depending only on $\Omega$, $p$ and $\nu$, where $\operatorname{Id}$ is the Beltrami-Klein identity map as defined in \eqref{defn.BK.identity}.
	By \eqref{quermass.formula.euclidean},
	\begin{align}\label{temp.sing.euclidean2}
		M(\operatorname{Id}(\Gamma)) - M(\operatorname{Id}(\Gamma^{-}(\nu,d))) = \int_{\widetilde{\mathcal{L}}} d L_2(\RR^3) ,
	\end{align}
	where $\widetilde{\mathcal{L}} = \operatorname{Id}(\mathcal{L})$ is the set of all the totally geodesic planes in $\RR^3$ that intersect $\operatorname{Id}(\Omega)$ but do not intersect $\operatorname{Id}(\Omega^{-}(\nu,d))$, and $d L_2(\RR^3)$ is the natural (invariant) measure on $\mathcal{L}_2(\RR^3)$. 
	
	Take the origin $O$ to be $\operatorname{Id}(p)\in B_{\RR^3}(1) \subset \RR^3$, and take the coordinates such that $\operatorname{Id}_{\star}(\nu) = (0,0,1)$ in $T_{\operatorname{Id}(p)}\RR^3$, then by \eqref{invariant.measure.RR3},
	\begin{align}\label{sing.invariant.measure}
		 dL_2(\RR^3) = 2 dR \wedge d\mu_{\SSS^2}.
	\end{align} 
	and $\operatorname{Id}(\Omega) \subset \RR^3_{-} := \{ (x,y,z)\in \RR^3: z <0 \}$.
	
	By  \eqref{temp.sing.euclidean2} and \eqref{sing.invariant.measure}, to prove  \eqref{temp.sing.euclidean1}, 
	it suffices to show an open subset of $dL_2(\RR^3)$-measure $c_1 d$ is contained in $\widetilde{\mathcal{L}}$. That is, there exist constants $c_1>0$ and $\delta>0$ depending only on $\Omega$, $p$ and $\nu$, such that 
	 \begin{align}\label{temp.sing.measure.positive}
	 	\operatorname{Plane} \left( B_{\SSS^2}(-\operatorname{Id}_{\star}\nu,\delta)\times(0, c_1 d) \right)
	 	=
	 	\operatorname{Plane} \left( B_{\SSS^2}((0,0,-1),\delta)\times(0, c_1 d) \right) \subset \widetilde{\mathcal{L}}, 
	 \end{align}
	where $B_{\SSS^2}((0,0,-1),\delta) $ denotes the open geodesic ball in $\SSS^2$ centered at $(0,0,-1)$ with radius $\delta$, and $\operatorname{Plane}$ is the parametrization of $\mathcal{L}_2(\RR^3)$ as defined in \eqref{grassmannian.paramter.R}.
	
	To proceed, we prove the following technical lemma concerning triangular pyramids.
	\begin{lem}\label{lemma.sing.pyramid}
		Let $K$ be a closed unbounded triangular pyramid (intersection of three linearly independent half-spaces) in $\RR^3$ with apex at the origin $O$. For any $\nu \in \operatorname{relint} \NNN(p,K)$ and any $d>0$, there exists $\delta >0$, depending only on $K$ and $\nu$, such that for all $\nu_1 \in B_{\SSS^2}(\nu, \delta)$, we have
		$$\operatorname{dist}_{\RR^3}(\operatorname{Plane}(\nu_1,0), K\cap \widetilde{U}(d)) > \frac{1}{2}d,$$
		where $\widetilde{U}(d): = \left\{ q \in \RR^3: \left\langle q, \nu \right\rangle \leq -d \right\} $. 
	\end{lem}
	\begin{proof}[Proof of Lemma \ref{lemma.sing.pyramid}]
		We take the unit outer normals to the three faces of $K$ and denote them by $\nu_1,\nu_2,\nu_3$ respectively. Then 
		$
			K = \left\{ q \in \RR^3: \left\langle q, \nu_{i} \right\rangle \leq 0 , \text{ for all } i \in {1,2,3} \right\},
		$
		and 
		\begin{align*}
			N(p,K) = \left\{ t_1\nu_1 + t_2\nu_2 + t_3\nu_3: t_{i} \geq 0, \text{ for all } i \in {1,2,3} \right\}.
		\end{align*}
		Since $\nu \in \operatorname{Int} N(p,K)$, there exist $\lambda_1,\lambda_2,\lambda_3 >0$ such that
		$\nu = \lambda_1\nu_1 + \lambda_2\nu_2 + \lambda_3\nu_3 $.
		Hence
		\begin{align*}
			K\cap \widetilde{U}(d) = 
			&\left\{ q \in \RR^3: \left\langle q, \nu_{i} \right\rangle \leq 0 , \text{ for all } i \in {1,2,3}, \text{ and } \left\langle q, \nu \right\rangle \leq -d \right\}
			\\
			=&
			\left\{ q \in \RR^3: \left\langle q, \nu_{i} \right\rangle \leq 0 , \text{ for all } i \in {1,2,3}, 
			\right.
			\\
			&
			\left.
			\text{ and } 
			\lambda_1\langle q,\nu_1\rangle
			+
			\lambda_2\langle q,\nu_2\rangle
			+
			\lambda_3\langle q,\nu_3\rangle
			 \leq -d 
			 \right\}.
		\end{align*}
		
		Hence if we take $\delta_1 = \frac{1}{2} \min\left\{ \lambda_1 ,\lambda_2,\lambda_3 \right\} $, then for all $q\in K\cap \widetilde{U}(d)$, and all $t_1,t_2,t_3$ such that $|t_i|<\delta_1$ for all $i\in {1,2,3}$,
		we have
		\begin{align*}
			&\langle q, (\lambda_1+t_1)\nu_1 + (\lambda_2+t_2)\nu_2 + (\lambda_3+t_3)\nu_3 \rangle
			\\
			=
			&
			(\lambda_1+t_1)\langle q,\nu_1\rangle
			+
			(\lambda_2+t_2)\langle q,\nu_2\rangle
			+
			(\lambda_3+t_3)\langle q,\nu_3\rangle
			\\
			\leq&
			(\lambda_1-\delta_1)\langle q,\nu_1\rangle
			+
			(\lambda_2-\delta_1)\langle q,\nu_2\rangle
			+
			(\lambda_3-\delta_1)\langle q,\nu_3\rangle
			\\
			\leq&
			\frac{1}{2} 
			\left( \lambda_1\langle q,\nu_1\rangle
			+
			\lambda_2\langle q,\nu_2\rangle
			+
			\lambda_3\langle q,\nu_3\rangle \right)
			=
			-\frac{1}{2} d
		\end{align*}
		Hence there exists $\delta >0$, depending only on $K$ and $\nu$, such that
		for all $\nu_1\in B_{\SSS^2}(\nu, \delta_1) $ and for all $q\in K\cap \widetilde{U}(d)$, we have $\langle q, \nu_1\rangle \leq-\frac{1}{2} d $, that is,
		$\operatorname{dist}(\operatorname{Plane}(\nu_1,0), K\cap \widetilde{U}(d)) \geq  \frac{1}{2}d$.
	\end{proof}

	Since $\operatorname{Id}_{\star}\nu \in N(p, \operatorname{Id}(\Omega_1))$ is a tangent vector in $\RR^3$, and $\operatorname{Id}(\Omega_{1})$ is a closed unbounded triangular pyramid in $\RR^3$, by Lemma \ref{lemma.sing.pyramid},  there exists $\delta >0$, such that for all $\nu_1 \in B_{\SSS^2}(\operatorname{Id}_{\star}\nu, \delta)$, 
	\begin{align}\label{temp.sing.dist.plane1}
		\operatorname{dist}_{\RR^3}
		(
		\operatorname{Plane}(\nu_1,0), 
		\operatorname{Id}(\Omega_1^{-}(\nu,d))
		)
		 >\frac{ 1}{2} \hat{d},
	\end{align}
	where $\hat{d} = \tanh(d)$ is the Euclidean distance in the image under the Beltrami-Klein identity map $\operatorname{Id}$, since the distance between $p$ and $\Omega_{1}^{-}(\nu,d)$ is $d$ in $\HH^3$. 
	
	Therefore, since $\Omega^{-}(\nu,d) \subset \Omega_1^{-}(\nu,d)$, by \eqref{temp.sing.dist.plane1}, we have
	$$ \operatorname{dist}_{\RR^3}(\operatorname{Plane}(\nu_1,0), \operatorname{Id}(\Omega^{-}(\nu,d))) > c_3 d, $$ for some constant $c_3>0$ depending only on $\Omega$, $p$ and $\nu$.
	Hence for any $R \in (0, \frac{1}{2} c_3 d)$, we have
	$$ \operatorname{dist}_{\RR^3}(\operatorname{Plane}(-\nu_1,R), \operatorname{Id}(\Omega^{-}(\nu,d))) > \frac{1}{2} c_3 d > 0 .$$
	
	Take $c_2:=\frac{1}{2} c_3$, then for all $\nu_1 \in B_{\SSS^2}(\operatorname{Id}_{\star}\nu, \delta)$ and $R \in (0, c_2 d)$, $\operatorname{Plane}(-\nu_1,R)$ intersects $\operatorname{Id}(\Omega)$ but does not intersect $ \operatorname{Id}(\Omega^{-}(\nu,d)) $. Hence 
	$$ \operatorname{Plane} \left( B_{\SSS^2}(-\operatorname{Id}_{\star}\nu,\delta)\times(-c_2 d, 0) \right) \subset \widetilde{\mathcal{L}}, $$
	which proves \eqref{temp.sing.measure.positive}, and thus \eqref{temp.sing.euclidean1} follows.
	  
	Now we will use \eqref{temp.sing.euclidean1} in $\RR^3$ to prove \eqref{temp.sing.hyperbolic1} in $\HH^3$.
	Since for any point $q \in (B_{\RR^3}(1), g^{BK}) = \HH^3$,  
	$\operatorname{dist}_{\RR^3}(\operatorname{Id}(O), \operatorname{Id}(q)) = 
	\tanh\left( \operatorname{dist}_{\HH^3}(O,q)) \right)
	$,
	we have
	\begin{align}\label{sing.mathcal.L}
		\widetilde{\mathcal{L}} = \operatorname{Id}(\mathcal{L}) = 
		\left\{ \operatorname{Plane}(\tau, \tanh(R)): \tau\in \SSS^2, R>0, \text{ such that } \operatorname{Plane}(\tau , R) \in \mathcal{L} .\right\}.
	\end{align}
	By \eqref{invariant.measure.HH3}, we know 
	\begin{align}\label{sing.invariant.measure.H}
		dL_2(\HH^3) = 2 (1+ 2\sinh^2(R)) dR \wedge d\mu_{\SSS^2}.
	\end{align}
	Combining \eqref{sing.mathcal.L} and \eqref{sing.invariant.measure.H}, since $\sinh(R)/R \to 1$ and $\tanh(R)/R \to 1$ as $R\to 0$, if $d$ is small enough, there exist constants $c$ and $c_1$ such that 
	\begin{align*}
		dL_2(\HH^3)(\mathcal{L}) 
		>
		c \cdot dL_2(\RR^3)(\operatorname{Id}(\mathcal{L})) > c_1 d,
	\end{align*}
	hence \eqref{sing.M.difference.hyperbolic} follows, which proves \eqref{temp.sing.hyperbolic1}.
	\end{proof}

	We now combine the two lemmas to finish the proof. By Lemma \ref{lem.diff.area} and Lemma \ref{lem.diff.total.mean.curv}, we choose any $\nu \in \NNN(p,\Omega_{1})\subset \NNN(p,\Omega)$. As $d\to 0^{+}$, there exists positive constants $c$ and $C$, depending only on $\Omega, p, \nu$, such that
	 \begin{align}
		\label{M-infinitesimal}
		&M(\Gamma) - M(\Gamma^{-}(\nu,d)) > c d, \\
		\label{S-infinitesimal}
		&S(\Gamma) - S(\Gamma^{-}(\nu,d)) \leq C d^2.
	\end{align} 
	
	We now apply geodesic normal flow to $\Gamma^{-}(\nu, d)$. From the proof of Proposition \ref{parallel.surface.property}, there exists a constant $T>0$ such that the outer parallel surface $X_T( \Gamma^{-}(\nu, d) ) $ at distance $T$ has the same surface area as $\Gamma$. Denote it by $\Gamma'$. Then $S(\Gamma') = S(\Gamma)$. 
	 
	 By Proposition \ref{parallel.surface.property} we have
	 $$ M(\Gamma')^2 - 16\pi S(\Gamma') - 4S(\Gamma')^2 = 
	 M(\Gamma^{-}(\nu,d))^2 - 16\pi S(\Gamma^{-}(\nu,d)) - 4S(\Gamma^{-}(\nu,d))^2, $$
	 Hence,
	 \begin{align*}
	 	&M(\Gamma')^2 - M(\Gamma)^2
	 \\= & M(\Gamma^{-}(\nu,d))^2 - M(\Gamma)^2 - 16\pi \left[S(\Gamma^{-}(\nu,d)) - S(\Gamma)\right] 
	 \\&- 4\left[S(\Gamma^{-}(\nu,d))^2 - S(\Gamma)^2\right].
	 \end{align*}
	 Substituting \eqref{M-infinitesimal} and \eqref{S-infinitesimal}, we obtain
	 $$ 
	 M(\Gamma')^2 - M(\Gamma)^2 
	 <  - \left[M(\Gamma^{-}(\nu,d))+ M(\Gamma)\right] c d + Cd^2
	 \leq -c_0 d + C d^2
	 $$
	 for constants $c_0, C$ depending only on $\Omega$, $p$ and $\nu$.
	 Therefore, for $d$ small enough, $S(\Gamma') = S(\Gamma)$ but $M(\Gamma') < M(\Gamma)$, so $\Gamma$ is not a minimizer, which contradicts the assumption. Hence if $\Gamma$ is a convex minimizer, then $\Gamma$ has no $0$-singular point.
	\end{proof}
	
	
	\begin{rem}\label{sec:5}
		\textbf{Open Questions.}
	Some further study directions to Santal\'{o}'s problem will be listed here. Clearly, the question of greatest interest is:
	\begin{ques}
		What is the shape of a minimizer to Santal\'{o}'s problem, that is, of a convex surface in $\HH^3$ that minimizes total mean curvature among convex surfaces with fixed area?
	\end{ques}
	
	The explicit form of the minimizer may be difficult to find; however, we may ask the following related questions to understand its geometry:
	\begin{ques}
		Let $\Gamma$ be a minimizer to Santal\'{o}'s problem then 
		\begin{itemize}
			\item Is $\Gamma$ strictly convex?
			\item Is $\Gamma$ regular? (Here, regular means that at every point the supporting plane is unique.)
			\item If $\Gamma$ is regular, is it $C^2$, or even $C^{\infty}$?
			\item Is $\Gamma$ axially symmetric?
		\end{itemize}
	\end{ques}

	By the arguments in Section \ref{sec:3}, one can prove that the minimizing surface $\Gamma$ to Santal\'{o}'s problem with area $S$ is contained in a geodesic ball $B( g(S))$ for some monotonically increasing function $g$ such that $\lim_{S \to 0^+} g(S) = 0$. Hence when $S$ is small enough, $\Gamma$ also lies in a small neighborhood of a point $p$ in $\HH^3$, and, as $S\to 0^+$, its geometric quantities can be arbitrarily close to its Euclidean image in $\RR^3$, under the Beltrami-Klein identity map $\operatorname{Id}$ defined in \eqref{defn.BK.identity} (where we take the origin to be $p$). Thus we may speculate when the area $S$ is small enough, the minimizer is close to a geodesic sphere, which is the minimizer to Minkowski Inequality \eqref{eq:minkowski0} in $\RR^3$.
	
	When $S$ is large enough, some mechanism we do not fully understand yet causes the flat double disks, whose mean curvature is concentrated entirely on the "edge", to have lower total mean curvature than spheres with the same area, and to become candidates for minimizers to Santal\'{o}'s problem. This phenomenon is possibly connected with the exponential growth of geometric quantities in hyperbolic spaces. As $r\to\infty$, for a geodesic sphere in $\HH^3$, its area and total mean curvature both grow at the rate of $e^{2r}$; in contrast, for a geodesic sphere in $\RR^3$, its area grows at the rate of $r^2$, and total mean curvature at the rate of $r$. This difference between hyperbolic and Euclidean geometry makes Santaló's problem more complicated. In view of the family of convex surfaces constructed in Section~\ref{sec:3}, when the area $S$ is large, one may speculate that the minimizer has singular points, is therefore non-smooth, and is close to a flat double disk.
	
	\begin{ques}
		Let $\Gamma(S)$ be a minimizer to Santal\'{o}'s problem with surface area $S$. Then 
		\begin{itemize}
			\item If $S\to 0$, will $\Gamma(S)$, after proper rescaling, converge to a sphere? 
			\item If $S\to \infty$, will $\Gamma(S)$, after proper rescaling,
			converge to a flat double disk?
		\end{itemize}
	\end{ques}
	
	\end{rem}
	
	\section*{Appendix: Computation of $\Phi_{1}(\Omega(r,\alpha), \Gamma_3(r,\alpha))$}
	\label{append.mean.curv.compute}
	In this appendix, we will recall the notion of curvature measures in hyperbolic spaces, and give a detailed computation of $\Phi_1(\Omega(r,\alpha), \Gamma_3(r,\alpha))$ defined in Section \ref{sec:3}.
	
	Let $\KK_0(\mathbb{H}^{n+1})$ be the set of compact convex sets in $\mathbb{H}^{n+1}$ with non-empty interior. For any $K\in \KK_0(\mathbb{H}^{n+1})$, and
	$\epsilon>0$, we define the set
	\begin{eqnarray*}
		K_{\epsilon}  &=&  \{ x\in \mathbb{H}^{n+1}: \operatorname{dist}_{\mathbb{H}^{n+1}}(x,K)\leq \epsilon\}.
	\end{eqnarray*}
	The map $f_{K}: \mathbb{H}^{n+1}\backslash K\rightarrow \partial K$ is defined by
	\begin{eqnarray}\label{defn.metric.proj}
		\operatorname{dist}_{\mathbb{H}^{n+1}}(f_{K}(x),x)=\operatorname{dist}_{\mathbb{H}^{n+1}}(x,K),
	\end{eqnarray}
	and is well-defined because $K$ is convex. We call $f_{K}$ the \emph{metric projection} onto $K$.
	For $\beta\subset \mathbb{H}^{n+1}$, we also define
	\begin{eqnarray*}
		M_{\epsilon}(K,\beta):=f_{K}^{-1}(\beta\cap \partial K)\cap (K_{\epsilon}\backslash K).
	\end{eqnarray*}
	Following Kohlmann \cite{kohlmann1991}, Allendoerfer \cite{allendoerfer1948}, define a Radon measure $\mu_{\epsilon}$ on the Borel $\sigma$-algebra of hyperbolic space
	$\mathcal{B}(\mathbb{H}^{n+1})$ by
	\begin{eqnarray*}
		\mu_{\epsilon}(K,\beta)=\operatorname{Vol}_{\mathbb{H}^{n+1}}(M_{\epsilon}(K,\beta)),
	\end{eqnarray*}
	Set $l_{n+1-r}(t):=\int_{0}^{t}\sinh^{n-r}(x)\cosh^{r}(x)dx, r=0,\cdots, n$. Then the following Steiner-type formula holds. See \cite{allendoerfer1948}, \cite{kohlmann1991}.
	\begin{equation}\label{steiner.hyperbolic} 
		\mu_{\epsilon}(K,\beta)=\sum_{r=0}^{n}l_{n+1-r}(\epsilon)\Phi_{r}(K,\beta),\qquad \forall\beta\in \mathcal{B}(\mathbb{H}^{n+1}).
	\end{equation}
	
	When $\eta=\partial K\cap \beta$ is a $C^{3}$ surface, $\Phi_{r}(K,\cdot)$ has the following simple expression:
	\begin{equation*}\label{curvature.measure.smooth} 
		\Phi_{r}(K,\beta)=\int_{\eta}\sigma_{n-r}^{K}(\kappa)d\mu_{g},
	\end{equation*}
	where $\mu_{g}$ is the surface area measure on $\partial K$ induced by $\operatorname{Vol}_{\mathbb{H}^{n+1}}$. $\sigma_{n-r}^{K}$
	is the $(n-r)$-th elementary symmetric function of the principal curvatures of $\partial K$.
	Generally, $\Phi_{n-r}(K,\cdot)$ is called the $r$-th curvature measure of the convex body $K$. In particular, $\Phi_{0}(K,\cdot)$ is called \emph{Gaussian curvature measure}, and $\Phi_{n-1}(K,\cdot)$ is called \emph{mean curvature measure}.
	For any convex surface $\Gamma$ in $\HH^{n+1}$ with enclosed convex body $\Omega = \operatorname{conv}(\Gamma)$, the total mean curvature of $\Gamma$ is given by $M(\Gamma) = \Phi_{n-1}(\Omega, \HH^{n+1}) = \Phi_{n-1}(\Omega, \Gamma)$. 
	
	For the convex surface $\Gamma(r,\alpha) \subset \HH^3$ defined in Section \ref{sec:3}, we will now use \eqref{steiner.hyperbolic} to compute $\Phi_1(\Omega(r,\alpha), \Gamma_3(r,\alpha))$. Since $\Gamma_3(r,\alpha)$ consists of two identical "circular edges", we define $$\Gamma_3^{+}(r,\alpha) : = \left\{x\in \Gamma_3(r,\alpha): \operatorname{Id}(x) \in B_{\RR^3}(1)\cap\{x_3>0\}\right\},$$
	where  $\operatorname{Id}$ is the Beltrami-Klein identity map defined in \eqref{defn.BK.identity}. Then it suffices to compute $\Phi_1(\Omega(r,\alpha), \Gamma_3^{+}(r,\alpha))$.
	
	To simplify the computation, we will use polar coordinates to model $\HH^3$. That is, we identify $\HH^3$ with $( U, \bar{g} )$ by the coordinate isometry $\operatorname{Pol}: \HH^3 \to ( U, \bar{g} )$, where
	\begin{align*}
		U:=\left\{ (\rho,\theta,\phi): \rho \geq 0, 0\leq \theta \leq \pi , 0 \leq \phi < 2\pi \right\},
	\end{align*} endowed with metric
	\begin{align*}
		\bar{g} = d\rho^2 + \sinh^2(\rho)d\theta^2 + \sinh^2(\rho) \sin^2(\theta) d\phi^2.
	\end{align*}
	$\bar{g}$ is well-defined in the subset of $U$ where $\rho>0$ and $0<\theta < \pi$, but can be naturally extended to all of $U$. Note that we will choose the origin such that the origin in polar coordinates $\operatorname{Pol}^{-1}\{\rho=0\} = A = \operatorname{Id}^{-1}(0,0,\hat{r}\cos(\alpha))$.
	
	Then we can place $\Gamma(r,\alpha)$ such that 
	\begin{align*}
		\Gamma_3^{+}(r,\alpha) = \operatorname{Pol}^{-1}\left\{ (\rho,\theta,\phi)\in U: \rho=r_1, \theta=\frac{\pi}{2} \right\},
	\end{align*}
	and
	\begin{align*}
		&M_{\epsilon}(\Omega(r,\alpha), \Gamma_3^{+}(r,\alpha))\\ = 
		\operatorname{Pol}^{-1}\{
		& (\rho,\theta,\phi)\in U: \text{There exists } 0\leq \gamma \leq \frac{\pi}{2}-\beta, 0<t < \epsilon, \text{ such that } 
		\\
		&
		\frac{\sinh(t)}{\sin(\frac{\pi}{2}-\theta)}  = \frac{\sinh(\rho)}{\sin(\frac{\pi}{2}+\gamma)},
		\text{ and } 
		\\
		&
		\cosh(\rho)  = \cosh(r_1) \cosh(t) - \sinh(r_1)\sinh(t) \cos(\frac{\pi}{2}+\gamma)
		 \}
	\end{align*}
	by the hyperbolic law of cosines and sines.
	Set $$N_{\epsilon}(r,\alpha) := \left\{x\in M_{\epsilon}(\Omega(r,\alpha), \Gamma_3(r,\alpha)): \operatorname{dist}_{\HH^3}(x,\Gamma_3^{+}(r,\alpha)) = \epsilon \right\}.$$
	Then for any $\epsilon>0$, $N_{\epsilon}(r,\alpha)$ is a $2$-dimensional submanifold of $\HH^3$ and we have 
	\begin{align*}
		\mu_{\epsilon}( \Omega(r,\alpha), \Gamma_3^{+}(r,\alpha)  ) = \operatorname{Vol}_{\HH^3}M_{\epsilon}(\Omega(r,\alpha), \Gamma_3^{+}(r,\alpha)) = 
	\int_{0}^{\epsilon}S(N_{t}(r,\alpha))dt .
	\end{align*}
	Clearly, for any $t>0$, $N_{t}(r,\alpha)$ is also given by
	\begin{align}\label{append.2.tmp.1}
		N_{t}(r,\alpha) = 
		\operatorname{Pol}^{-1}\{
		& (\rho,\theta,\phi)\in U: \text{There exists } 0\leq \gamma \leq \frac{\pi}{2}-\beta, \text{ such that } 
		\\
		\nonumber
		&
		\frac{\sinh(t)}{\sin(\frac{\pi}{2}-\theta)}  = \frac{\sinh(\rho)}{\sin(\frac{\pi}{2}+\gamma)},
		\text{ and } 
		\\
		\nonumber
		&
		\cosh(\rho)  = \cosh(r_1) \cosh(t) - \sinh(r_1)\sinh(t) \cos(\frac{\pi}{2}+\gamma)
		\}
	\end{align}
	We may express $\rho$ and $\theta$ implicitly as functions of $\gamma$, and use them to parametrize $N_{t}(r,\alpha)$, namely
	\begin{align}\label{append.2.tmp.2}
		N_{t}(r,\alpha) = \operatorname{Pol}^{-1}\{
		& (\rho_t(\gamma),\theta_t(\gamma),\phi):  0\leq \gamma \leq \frac{\pi}{2}-\beta, 
		0\leq \phi< 2\pi
		\},
	\end{align}
	where the functions $\rho = \rho_t(\gamma)$ and $\theta = \theta_t(\gamma)$ are determined by the defining conditions in \eqref{append.2.tmp.1}.
	By \eqref{append.2.tmp.2}, the area element of $N_{t}(r,\alpha)$ is given by
	\begin{align}\label{append.2.temp.3}
		d\mu_{N_{t}(r,\alpha)}  =  \sinh(\rho_t(\gamma)) \sin(\theta_t(\gamma)) \sqrt{ (\rho_t'(\gamma))^2 + \sinh^2(\rho_t(\gamma)) (\theta_t'(\gamma))^2 }d\gamma d\phi.
	\end{align}
	Differentiating the two equations in \eqref{append.2.tmp.1}, we have
	\begin{align}\label{append.2.tmp.4}
		\sinh(\rho_t(\gamma))\rho_t'(\gamma) =  \sinh(r_1)\sinh(t) \cos(\gamma)
	\end{align}
	and
	\begin{align}\label{append.2.tmp.5}
		\cosh(\rho_t(\gamma)) \rho_t'(\gamma) \cos(\theta_t(\gamma)) 
		-
		\sinh(\rho_t(\gamma)) \sin(\theta_t(\gamma)) \theta_t'(\gamma)
		=
		\sinh(t) \sin(\gamma).
	\end{align}
	Substituting \eqref{append.2.tmp.5} and \eqref{append.2.tmp.4} into \eqref{append.2.temp.3} and using \eqref{append.2.tmp.1}, we have
	\begin{align*}
		d\mu_{N_{t}(r,\alpha)}  =  \sinh t
		\left(
		\sinh r_1\cosh t
		+
		\cosh r_1\sinh t\sin\gamma\right)
		d\gamma d\phi
	\end{align*}
	and hence
	\begin{align*}\label{append.2.eq.S.N_t}
		&S(N_{t}(r,\alpha))
		\\
		=&
		\int_{0}^{\frac{\pi}{2}-\beta} d\gamma \int_{0}^{2\pi} d\phi 
		\sinh t\left(
		\sinh r_1\cosh t
		+
		\cosh r_1\sinh t\sin\gamma\right)
		\\
		\nonumber
		=
		&2\pi  \sinh(r_1) \sinh(t) \cosh(t) (\frac{\pi}{2} - \beta)
		+
		2\pi \cosh(r_1) \sinh^2(t) (1-\sin(\beta)).
	\end{align*}
	Hence
	\begin{align*}
		&M_{\epsilon}(\Omega(r,\alpha), \Gamma_3(r,\alpha)) \\
		= 
		&2 M_{\epsilon}(\Omega(r,\alpha), \Gamma_3^{+}(r,\alpha))\\
		=
		&4 \pi  \sinh(r_1) (\frac{\pi}{2} - \beta) \int_{0}^{\epsilon} \sinh(t) \cosh(t) dt \\
		&
		+
		4\pi \cosh(r_1)  (1-\sin(\beta)) \int_{0}^{\epsilon} \sinh(t)^2 dt .
	\end{align*}
	Comparing with \eqref{steiner.hyperbolic}, we obtain $\Phi_{1}(\Omega(r,\alpha), \Gamma_3(r,\alpha)) =4 \pi  \sinh(r_1) (\frac{\pi}{2} - \beta)  $.
	
	\addtocontents{toc}{\protect\setcounter{tocdepth}{0}}
	\section*{Acknowledgments}
	I would like to thank my supervisor Pengfei Guan for introducing this problem to me and for many inspiring discussions, and Xiaodong Yang for useful comments on the numerical computation.
	
	\textbf{Statement on the use of AI}:  The author proved the math content of this paper approximately one year before this revision. During the revision, the authors used ChatGPT to assist with literature review and to check the wording. The proof of  \ref{thm:counterexample} was originally via numerical computation of Matlab, but we also used ChatGPT to give a more elementary proof.  
	The
	author take full responsibility for the mathematical content and the final text.
	
	\textbf{Data Availability}: No datasets were generated or analyzed during the current study.
	
	\textbf{Conflicts of Interest}: The author declares that there is no conflict of interest.
	
	\addtocontents{toc}{\protect\setcounter{tocdepth}{1}}
	\bibliographystyle{amsplain}

	\end{document}